\documentclass[12pt]{amsart}
   \newtheorem{thm}{Theorem}[section]
  \newtheorem{prop}[thm]{Proposition}
  
  \newtheorem{cor}[thm]{Corollary}
  \theoremstyle{definition}
 \theoremstyle{remark}
 \newtheorem{rmk}[thm]{Remark}
  \newtheorem{fact}[thm]{Fact}
 \newtheorem{examp}[thm]{Example}

 \usepackage{amsfonts}
 \usepackage{amsmath}
 \usepackage{calligra}
 \usepackage{amssymb}
 \usepackage{verbatim}
 \usepackage[colorlinks]{hyperref}
 \usepackage[T1]{fontenc}
 \usepackage{mathrsfs}
\usepackage{color}
 \usepackage[utf8]{inputenc}
 \usepackage[makeroom]{cancel}
 \usepackage{tikz-cd}
 \input xy
 \xyoption{all}
 \newenvironment{psmallmatrix}
   {\left(\begin{smallmatrix}}
 {\end{smallmatrix}\right)}
   \usepackage{mathtools}

 \newcommand{\g}{\mathfrak{g}}
 \newcommand{\h}{\mathfrak{h}}
 \renewcommand{\k}{{\mathfrak k}}
 \renewcommand{\l}{\mathfrak l}
 
 \newcommand{\p}{\mathfrak p}
 \newcommand{\q}{\mathfrak q}
 
 \newcommand{\z}{\mathfrak z}
 \renewcommand{\t}{{\mathfrak t}}
 \renewcommand{\u}{{\mathfrak u}}
 
 \newcommand{\C}{{\mathbb C}}

 \numberwithin{equation}{section}
\usepackage{etoolbox}
 \DeclareFontFamily{U}{mathx}{\hyphenchar\font45}
 \DeclareFontShape{U}{mathx}{m}{n}{
       <5> <6> <7> <8> <9> <10>
       <10.95> <12> <14.4> <17.28> <20.74> <24.88>
       mathx10
       }{}
 \DeclareSymbolFont{mathx}{U}{mathx}{m}{n}
 \DeclareFontSubstitution{U}{mathx}{m}{n}
 \DeclareMathAccent{\widecheck}{0}{mathx}{"71}
 \DeclareMathAccent{\wideparen}{0}{mathx}{"75}

\begin{document}
 \title{Branching laws and a duality principle, Part I}
\author{  Bent {\O}rsted,   Jorge A.  Vargas}
 \thanks{Partially supported by Aarhus University  (Denmark),  CONICET
 (Argentina)  }
 \date{\today }
 \address{ Mathematics Department, Aarhus University, Denmark; FAMAF-CIEM, Ciudad Universitaria, 5000 C\'ordoba, Argentine}
 \email{orsted@math.au.dk,   vargas@famaf.unc.edu.ar}

 \begin{abstract}
  For a semisimple Lie group $G$ satisfying the
 equal rank condition, the most basic family of unitary
 irreducible representations is the Discrete Series found
 by Harish-Chandra. In this paper,  we continue our study of the branching laws for Discrete Series when restricted to a subgroup $H$ of the same
type by use of integral and differential operators in combination with our previous duality principle.  Many results are presented in generality, others are shown in detail for Holomorphic Discrete Series.
 \end{abstract}
 \maketitle
 \markboth{{\O}rsted- Vargas}{Duality, branching}
{\centering\footnotesize  We are delighted to make this paper part of a tribute to Karl-Hermann Neeb,\\ for his
indefatigable dedication to representation theory\\ and Lie theory.\par}
 \tableofcontents
 \section{Introduction}

For a unitary representation $\pi$ of a Lie group $G$ in a Hilbert space $V$ we have
for each one-parameter subgroup $exp(tX)$ a unitary one-parameter group of
operators $U(t) = \pi(exp(tX))$, and this is given by the Fourier transform of a spectral
measure  on the real line with values in the projections in $V$. It is a fundamental question
to know this spectral measure and in particular the spectrum, i.e. its support. When the
generator of $U(t)$ corresponds to a physical interpretation such as energy, then one
thinks of the notion of positive energy as  related to the spectrum being bounded below.
This notion has been studied in great detail by Karl-Hermann Neeb in connection with
unitary  highest weight representations, where also the connection to causality and
field theory is central. In the works by Harish-Chandra this class of representations
was introduced as holomorphic Discrete Series of a semi-simple Lie group $G$ and related
to analysis and holomorphic vector bundles on the Riemannian Symmetric space $G/K$. Now in general for a unitary
representation of a Lie group $G$ it is of interest to find the spectrum of its restriction
to a subgroup $H$ (just as for one-parameter subgroups) - the branching law - and this will reveal much about the nature of the
representation. One important aspect is that one needs good models of  the Hilbert space
in order to carry out the restriction explicitly, and again typical models are in homogeneous
vector bundles over $G/K$ in the semi-simple situation, or equivalent versions (via
parallelization of the bundles) in vector-valued function on $G/K$. Whereas the overall principles
are simple, the computations involve the structures of $G$ and $H$ in suitable coordinates
using the root systems. The coordinates are relevant since we want to represent the
branching laws using integral kernel operators and also differential operators.

  In this paper we continue our study of the Discrete Series of a connected linear semisimple Lie group $G$,
 namely the unitary irreducible representations $\pi$ arising as closed subspaces of the left regular representation
 in $L^2(G)$. Here $G$ and a maximal compact subgroup $K$ have the same rank and Harish-Chandra gave
 a parametrization of such $\pi$, the so-called Discrete Series of $G$.

 Our aim is to understand the restriction
 of the Discrete Series of $G$ to a symmetric subgroup $H$ of $G$ in the admissible case, namely when
 a $\pi$ restricted to $H$ is a direct sum  of irreducible subspaces, each of them Discrete Series $\rho$ of $H$ with
 finite multiplicities. There are a number of techniques that we use here, firstly the theory of reproducing kernels,
 corresponding to convenient models of the Discrete Series. There are two kinds of operators we analyze, namely
 \begin{itemize}
 \item symmetry breaking operators, i.e. $H$ equivariant linear maps from $\pi$ to the individual $\rho$
 \item holographic operators, i.e. $H$-equivariant linear maps from the individual $\rho$ to $\pi$
 \end{itemize}
 As it turns out differential operators play a key role as symmetry breaking operators, and we shall explain
 this using homogeneous vector bundles over $G/K$. Another technique is that of pseudo-dual pair, a notion  introduced in \cite{OS}, and
 useful in our situation is that  we associate to $H$ another symmetric subgroup $H_0$ of $G$, and there
 is a corresponding duality theorem, which in a sense reduces the branching law (the explicit decomposition
 of $\pi$) to a branching law for $H_0$ under its maximal compact subgroup.

 We treat in special detail the case of $G/K$ of Hermitian type and the holomorphic Discrete Series; here we give
 a simpler proof of the duality theorem and also more information about the differential operators and
 kernel operators arising as holographic operators. This case has been treated earlier using the models
 specific to this case, namely seeing $G/K$ as a bounded symmetric domain; we are particularly
 interested in the nature of the symmetry breaking differential operators, in particular whether they
 are purely tangential or contain normal derivatives.

 In order to describe the  main results we set up some notation.

  From now on,  $G$ is a connected semisimple matrix Lie group. $K$ is a maximal compact subgroup of $G$. $\theta$ denotes the Cartan involution associated to $K$. $Lie(G)=\g=\k+\p$ is the associated Cartan decomposition. $\sigma$ a involution in $G$ that commutes with the Cartan involution $\theta$. $H:=(G^\sigma)_0$ the connected component of the identity of $G^\sigma$. The associated subgroup to $H$ is $H_0:=(G^{\sigma \theta})_0$. $L:=K\cap H$ is a maximal compact subgroup of both $H,H_0$.  $(\tau, W)$ (resp. $(\sigma, Z)$) is an irreducible representation of $K$ (resp. $L$). $dg$ denotes a Haar measures  in $G$. We recall the space \begin{multline*} L^2(G\times_\tau W)=\{ f :G\rightarrow W: f(xk)=\tau (k^{-1})f(x),\\ k\in K, x\in G,  \int_G \Vert f(x) \Vert_W^2 dx <\infty \}. \end{multline*}   The left action on functions is denoted by $L_\cdot^G, L_\cdot$. Similarly, we define $L^2(H\times_\sigma Z)$. An irreducible unitary representation of $G$ is square integrable, equivalently a Discrete Series representation, if each of its matrix coefficient is square integrable with respect to Haar measure. It is a Theorem that square integrable representations are determined by their lowest $K$-type, in the sense of Vogan,  and that the lowest $K$-type has multiplicity one when we restrict a given square integrable representation to the maximal compact subgroup $K$. Thus, Frobenius reciprocity let as conclude that a Discrete Series representation of lowest $K$-type $(\tau, W)$ has multiplicity one in $L^2(G\times_\tau W)$. Let   $H^2(G,\tau)\subset  L^2(G\times_\tau W)$ denote the unique closed linear subspace that affords the square integrable representation of lowest $K$-type $(\tau, W)$. It can be shown that $H^2(G,\tau)$ is a  eigenspace of the Casimir operator $\Omega_\g$. We denote by $\lambda_\tau$ the value of such eigenvalue.    Hence, a key property of $H^2(G,\tau)$ is being a reproducing kernel subspace. We denote the corresponding reproducing kernel by $K_\tau : G\times G \rightarrow Hom_\C(W,W)$. We have $K_\tau(x,y)=K_\tau(y,x)^\star$, $K_\tau(x,\cdot)^\star \in H^2(G,\tau)$, the orthogonal projector onto  $H^2(G,\tau)$ is the integral operator $f\mapsto \int_G K_\tau(y,\cdot) f(y)dy$, and   for every $f\in H^2(G,\tau), w \in W$ the identity $(f(x),w)_W=\int_G(f(y),K_\tau(y,x)^\star w)_W dy $ holds. Also,  $ \forall w \in W,  K_\tau(e,\cdot)^\star w \in H^2(G,\tau)[W] $.   Similarly, for a irreducible representation $(\sigma, Z)$ of $L$,  we consider   $H^2(H,\sigma)\subset L^2(H\times_\sigma Z)$.

 To continue, we describe the main results and  a resume  of the paper. Before we proceed, we would like to point out that the style of this note is inspired by K.-H. Neeb in the sense of aiming first a presentation with a high degree of generality followed by particular (and interesting) more special cases.    In Proposition~\ref{prop:propertiesksc}, Proposition~\ref{prop:propertiesholo}, we  recall scattered results on basic properties of intertwining linear operators between two representations modeled on reproducing kernel spaces. In  \ref{sub:gener1},\ref{sub:gener2}, we recall on the general notion of differential operator in our context.   With respect to
    the main results   a description is  as follows: In Theorem~\ref{prop:Sisnicediff} we generalize a result of Helgason, we show whenever a symmetry breaking operator is equal to the restriction of a plain differential operator, we may replace a differential operator by a canonical differential operator between the  vector bundles  that describe the representations. In the same section, we carry out in  thorough detail   the problem of writing a symmetry breaking operator as a differential operator given  the kernel and viceversa.   In Theorem~\ref{thm:xx} we continue the study of  the relation between symmetric breaking operators for different realizations of the same representation, among them, maximal globalization, Hilbert space realization, Harish-Chandra module realization.  In section 3 we consider the problem of representing a symmetry breaking operator as generalized gradient operator, the obtained result is valid for an arbitrary symmetric pair $(G,H)$ and arbitrary symmetry breaking operator for  an arbitrary  $H$-admissible Discrete Series representation for $G$.  To follow, in section 4, we present a new proof of the duality Theorem for the holomorphic setting,  Theorem~\ref{prop:dualfirstv}, Theorem~\ref{prop:firstversion},  Theorem~\ref{prop:secondver}, Theorem~\ref{prop:direciso}.  Our proof is of algebraic nature. In order to carry out the proof, we introduce the subspace $\mathcal L_{W,H}^c$   to be the linear span of the totality of  subspaces of the representation that affords the lowest $L$-type of each irreducible $H$-subrepresentation, and, $\mathcal U(\h_0)W  :=$the $\mathcal U(\h_0)$ submodule spanned by the subspace  that realizes the lowest $K$-type  of the representation. Then,  in Proposition~\ref{prop:D},  we explicit find an isomorphism $\mathcal L_{W,H}^c \stackrel{D}\equiv \mathcal U(\h_0)W$. In Proposition~\ref{prop:qisd} we show that $D$ may be replaced by the orthogonal projector onto $\mathcal U(\h_0)W$. In Theorem~\ref{prop:kernforhol} we present a new formula for holographic operators based on the reproducing kernel for the initial space. As a consequence, we show   a "separation of variable" formula for the kernel of holographic   as well for symmetry breaking operators, we believe this formula is valid under more general hypothesis, see Corollary~\ref{cor:symmsplit} and the previous corollaries. In section 5 we present an overview of some of the results that will appear in Part II. The results are on a quite careful
    analysis of  when all "first order" symmetry  breaking operators are represented by normal derivative differential operators, later on, we study, when the totality of symmetry breaking operators are represented by normal derivative operators. That is, following \cite[Proposition 6.3]{OV3},  we analyze the equality $\mathcal L_{W,H}^c =\mathcal U(\h_0)W$. Aspects of the notion of duality via dual pairs (also called pseudo dual pairs) has been considered in earlier papers by several authors, notably Jakobsen-Vergne, Kobayashi-Pevzner, Speh and Nakahama
 ; we believe it is useful for branching theory in quite general situations.
Some natural open problems arise from our study, for example we mention a few:

(1) the explicit nature of the
branching laws in terms of the reproducing kernels for  the representations;
since these are essentially generalized hypergeometric functions with parameters
coming from the large group (resp. the subgroup), there will be in the admissible case
an explicit identity expressing the kernel for the large group as a sum over kernels
for the subgroup.

 (2) Since other families of unitary representations have aspects in common with
the discrete series (such as reproducing kernels given in terms of certain matrix coefficients)
we may have similar results for these.

 (3) For the non-admissible cases of branching laws
(with both a discrete and a continuous spectrum) it is still interesting to consider the
symmetry-breaking operators for the discrete spectrum - here we have some conjectures, but
the theory remains incomplete.

 \subsection{Notation}\label{sub:not}  For unexplained concepts please consult  the partial list   below as well as  the Section Partial list of symbols and definitions, or \cite{OV2}\cite{OV3}.  The complexification of a real vector space $V$ is denoted $V_\C$. Quite often we are somewhat sloppy in writing the complex subindex $\C$. A {\it symmetry breaking operator}, that is, a continuous $H$-intertwining linear map from a representation of $G$ into  a representation of $H$ is denoted by $S$, whereas, a {\it holographic operator}, that is, a continuous $H$-intertwining linear map from the representation of $H$ into  a representation of $G$ is denoted by $T$. For a module $M$ and a simple submodule $N$, $M[N]$ denotes the {\it isotypic component} of $N$ in $M$. That is, $M[N]$ is the sum of all irreducible submodules isomorphic to $N.$ If topology is involved, we define $M[N]$ to be the closure of $M[N].$ $Hom_H(V, U)$ denotes the linear space of continuous $H$-intertwining linear maps from $V$ into $U$. For a finite dimensional representation $Z$   and a representation $U$ of a compact Lie group  $L$, $Hom_L(Z,U)$ is equal to the set of linear maps from $Z$ into $U$ that intertwines the action of $L$.

 \section{Analysis of  symmetry breaking (holographic) operators in the symmetric space model.}\label{sec:sym(holo)}
 \subsection{Symmetric space model for Discrete Series} We recall the notation $G,K,H,L$, $(\tau, W)$, $(\sigma,Z)$, $ H^2(G,\tau), H^2(H,\sigma)$. Our hypothesis for subsections 2.1, 2.2, 2.3   is: $(G,H)$ is a {\it arbitrary reductive pair}. From section 2.4 on, we further assume $(G,H)$ is a symmetric pair.  Let $c_\tau :G\times G/K \rightarrow GL(W),  c_\sigma :H\times H/L \rightarrow GL(Z) $ be cocycles. Thus, $c(ab,x)=c(a,bx)c(b,x)$ and $c_\tau(k,eK)=\tau(k), c_\sigma(l,eL)=\sigma(l), k\in K,l \in L$.  Let $dm_{G/K}$ denote a $G$-invariant measure on $G/K$ adjusted so that \begin{equation*}\int_G f(x)dx=\int_{G/K} \int_K   f(xk) dk dm_{G/K}(xK).\end{equation*}We recall the space (for a reference see  \cite{Hi} \cite{Neeb},   \cite{OV2}) \begin{multline}   L_{c}^2(G/K,W):$ $=L_{c_\tau}^2(G/K,W) :=    \{ f:G/K\rightarrow W : \\  \Vert f\Vert_{L_{c}^2(G/K,W)}^2 := \phantom{xxxxxxxxxxxxxxxxxxxxxxxxxxxxxxxxxx} \\ \phantom{xxxxxx}     \int_{G/K}  (\big(c_\tau(x,eL)c_\tau(x,eL)^\star \big)^{-1}  f(x), f(x))_W dm_{G/K}(x) <\infty \}.\end{multline} and the unitary  representation $\pi_{c_\tau}$ of $G$ on $L_{c_\tau}^2(G/K,W) $ defined by means of the equality $\pi_{c_\tau}(x)(f)(yK)=c_\tau(x^{-1},yK)^{-1}f(x^{-1}yK)$. We also obtain $L_{c_\sigma}^2(H/L,Z), \pi_{c_\sigma} $ etc. \\ Roughly speaking, the space $L_{c_\tau}^2(G/K,W)$ may be thought as the space of square integrable functions with respect to the "measure" $\mu_{c_\tau}:=  (c_\tau(x,e)c_\tau(x,e)^\star)^{-1} dm_{G/K}(x)$.  Henceforth,   $H_{c_\tau}^2(G/K,W)$ denotes the $\lambda_\tau$-eigenspace of $\pi_{c_\tau}(\Omega_\g)$ in $L_{c_\tau}^2(G/K,W)$.  The map    $f \mapsto E_{c_\tau}(f)(\cdot):=c_\tau(\cdot,e)f(\cdot)$   is a unitary equivalence between \\ $(L_\cdot,L^2(G \times_\tau W))$ and $(\pi_{c_\tau}, L_{c_\tau}^2(G/K,W))$,\\ carrying $ H^2(G,\tau) $ onto $  H_{c_\tau}^2(G/K,W) $.  Thus, $H_{c_\tau}^2(G/K,W)$ is a reproducing kernel space. We denote the corresponding reproducing kernel by \begin{equation} \label{eq:ktaucktau} K_\tau^c(xK,yK)=c_\tau(y,eK)K_\tau (x,y)c_\tau(x,eK)^{\star}\end{equation} We have, \begin{equation*} K_\tau^c(hy,hx)=c_\tau (h,x)K_\tau^c (y,x) c_\tau (h,y)^\star, \, \forall h \in G, x, y \in G/K. \end{equation*} Since, $K_\tau (x,x)=d(\pi_{c_\tau})   I_W$,   for every  $ x\in G,\, \text{it\, holds,}$ \\ \phantom{xxxxxxxxxxxxxx} $ K_\tau^c (xK,xK) =d(\pi_{c_\tau}) \, c_\tau (x,e) c_\tau (x,e)^\star . $

 \smallskip

 \subsection{Symmetry breaking operators}
  Next, we fix a continuous  symmetry breaking operator

 \phantom{xxxxxxxxx} $S: H_c^2(G/K,W) \rightarrow  H_c^2(H/L,Z)$.

   Since the target space is a reproducing kernel space,  there exists $K_S^c : G/K \times H/L \rightarrow Hom_\C (W,Z)$ so that for  $z\in Z, h\in H$,  $K_S^c(\cdot,hL)^\star z \in  H_c^2(G/K,W)$ and   for all $ \, F \in H_c^2(G/K,W), h \in H, z\in Z$, it holds \begin{equation}\label{eq:kerofSc} (S(F)(hL),z)_Z=(F,K_S^c(\cdot,hL)^\star z)_{L_{c_\tau}^2(G/K,W)}. \end{equation}  Similarly, for $S^\star$ and each $w\in W$, we have a kernel $K_{S^\star}^c(\cdot,xK)^\star w \in H_c^2(H/L,Z)$ such that  for all $  g_1 \in H_c^2(H/L,Z), x \in G$, the equality   $(S^\star(g_1)(xK),w)_W=(g_1, K_{S^\star}^c(\cdot,xK)^\star w)_{L_{c_\sigma}^2(H/L,Z)}$   holds. To follow, we show a result analogous to some statements in  \cite[Proposition 3.7]{OV2},    and, in the proof of  \cite[Lemma 4.2]{OV2}.
 \begin{prop} \label{prop:propertiesksc} Under the assumptions of the previous paragraph,   the following facts hold for an intertwining continuous linear map  $$S: H_c^2(G/K,W) \rightarrow  H_c^2(H/L,Z).$$
 \phantom{xxx} $ h,y \in H, x \in G,  z\in Z, w\in W, F \in H_c^2(G/K,W), g\in L_c^2(H/L,Z).  $
 \begin{enumerate}
 \item  $ S(F)(hL)$ \\ \phantom{xxx} $=\int_{G/K} K_S^c(xK,hL)\big(c_\tau(x,e)c_\tau(x,e)^\star \big)^{-1}F(xK) dm_{G/K}(xK). $
 \item $ K_S^c(xK,hL)w=K_{S^\star}^c(hL,xK)^\star w=S (K_\tau^c(\cdot,xK)^\star w)(hL).  $
 \item $ K_{S^\star}^c(   xK, eL ) w= S^\star (K_\sigma^c(\cdot,eL)^\star w)(xK).  $
 \item $ K_S^c(hxK,hyL)=c_\sigma (h,yL)K_S^c (xK,yL) c_\tau (h,xK)^\star.$
 \item $ K_{S^\star}^c(  hL, xK) z= K_{S}^c(xK, hL)^\star z=S^\star(K_\sigma^c(\cdot,hL)^\star z)(xK). $
    \item    The\, linear\, map,  \\ \phantom{xxxxx}   $ Z\ni z \mapsto K_S^c(\cdot,eL)^\star z =K_{S^\star}(eL,\cdot)z \in H_c^2(G/K,W),$  \\ is  a \, L-map .    Moreover, for each  $z\in Z , \\ \phantom{xxx} K_S^c(\cdot,eL)^\star z =K_{S^\star}(eL,\cdot)z \in H_{c_\tau}^2(G/K,W)[H_{c_\sigma}^2(H/L,Z)][Z]$.
 \item $ (S (K_\tau^c(\cdot,xK)^\star w)(hL),z)_Z=(w,S^\star (K_\sigma^c(\cdot,hL)^\star  z)(xK))_W   .$
 \item $  \,\,\, ( S^\star (g)(xK), w)_{W}=(g,S (K_\tau^c(\cdot,xK)^\star w)_{L_c^2(H/L,Z)} . $
     \end{enumerate}\end{prop}
  If we replace the   space $H^2(H,\sigma)$ by $L^2(H\times_\sigma Z)$,  the fourth identity might not hold!
 \begin{proof}A justification of all the statements but the fifth,  is in \cite{Neeb}. We carry out the necessary computation to justify  the first and the fifth  assertions.
   \begin{multline*}   (S(F)(hL),z)_Z \\  \phantom{xxxx} =\int_{G/K} ((c_\tau(x,e)c_\tau(x,e)^\star)^{-1}F(xK), K_S^c(xK,hL)^\star z)_W dm(xK) \\ \phantom{xxx}  =\int_{G/K} (K_S^c(xK,hL)(c_\tau(x,e)c_\tau(x,e)^\star)^{-1}F(xK),   z)_Z dm(xK) \\   =(\int_{G/K}  K_S^c(xK,hL)(c_\tau(x,e)c_\tau(x,e)^\star)^{-1}F(xK)  dm(xK) ,   z)_Z. \end{multline*}

  For the fifth statement, we apply the third statement  to $l\in L, x=eL=o=lo$,   we obtain $K_S^c (ly,o)^\star =c_\tau(l,y) K_S^c (y,o)^\star c_\sigma(l,o)^\star .$ Since, $c_\sigma(l,o)^\star=\sigma (l^{-1})$, we have $c_\tau(l,y)^{-1}K_S^c (ly,o)^\star =  K_S^c (y,o)^\star \sigma(l^{-1})  .$ Thus, the map $z \mapsto K_S^c ( \cdot,o)^\star z$ is a $L$-map. \\ Now,    for $G_1 \in (H_{c_\tau}^2(G/K,W)[H_{c_\sigma}^2(H/L,Z)])^\perp $, due that $H_{c_\sigma}^2(H/L,Z)$ is $H$-irreducible, we have $S(G_1)=0$. Hence, $0=(S(G_1)(hL),z)_Z =(G_1,K_S^c(\cdot, hL)^\star z)_{H_c^2(G/K,W)}$
  yields  \\$K_S^c(\cdot, hL)^\star z \in H_{c_\tau}^2(G/K,W)[H_{c_\sigma}^2(H/L,Z)]$,  the previous computation gives $K_S^c(\cdot, hL)^\star z \in  H_{c_\tau}^2(G/K,W)[H_{c_\sigma}^2(H/L,Z)][Z]$ and we have verified $(5)$. $(6),(7)$ readily follows.   \end{proof}

 \medskip

  The   diagram below shows how to transfer intertwining operators $S$ in the functions on  group model of representations   to operators $S^c$ in the symmetric space  model. The corresponding relation between the kernels $K_S$, $K_{S^c}^c$ is \begin{equation} K_{S^c}^c(xK,hL)=c_\sigma (h,e)K_S (x,h) c_\tau (x,e)^\star, \, \forall h \in H, x \in G.\end{equation}    \cite[Proposition 3.7]{OV2} yields $K_{S^c}^c$ is a smooth function.

 \xymatrix{ \phantom{xx} & {H^2(G,\tau)} \ar[r]^{\cong}  \ar[d]_{S}     & {H_{c_\tau}^2(G/K,W) } \ar[d]^{S^c} & f \ar[r]^{E_{c_\tau}} & c_\tau(\cdot,e)f(\cdot)  \\
    \phantom{xx} & {H^2(H,\sigma)}   \ar[r]^{\cong} & {H_{c_\sigma}^2(H/L,Z)} & g \ar[r]^{E_{c_\sigma}} & c_\sigma(\cdot,e)g(\cdot). }

 \begin{rmk}Let $r:H^2(G,\tau) \rightarrow C^\infty(H\times_\tau W)$ denote the restriction map. In \cite{OV} it is shown the image of $r$ is contained in $L^2(H\times_\tau W)$ and $r$ is $(2,2)$   continuous. It readily follows that the corresponding map $r^c :H_{c_\tau}^2 (G/K, W)\rightarrow L_{c_\tau}^2(H/L,W)$ is again the restriction map. We will apply this observation for the restriction $r_0 : H^2(G,\tau) \rightarrow L^2(H_0 \times_\tau W)$.
 \end{rmk}

 {\bf Note:} Under the assumption of $res_H(H^2(G,\tau))$ is $H$-admissible, in \cite{OV2}, we analyze the  orthogonal projectors $P_{\tau, \sigma}, (resp. P_{\tau, \sigma}^c$)   onto  the    isotypic components $H^2(G,\tau)[H^2(H\times_\sigma Z)]$\\ (resp. $H_{c_\tau}^2(G/K,W)[H_{c_\sigma}^2(H/L, Z)]$)  we state without proofs  the following relation among the respective kernels. \begin{gather*} K_{\tau,\sigma}^c(xK,yK)=c_\tau(y,e)K_{\tau,\sigma}(x,y)c_\tau(x,e)^\star. \\
 K_{\tau,\sigma}(x,y)=\Theta_{H^2(H\times_\sigma Z)}(h\mapsto K_\tau (h^{-1}x,y)). \\ K_{\tau,\sigma}^c( x,y)=\Theta_{H^2(H\times_\sigma Z)}(h\mapsto K_{\tau}^c (h^{-1}x,y)). \end{gather*} Here, $\Theta_{H^2(H\times_\sigma Z)}$ is the Harish-Chandra distribution character of $H^2(H\times_\sigma Z)$.
 \subsection{Holographic operators} By definition, holographic operators are the continuous $H$-intertwining linear maps \begin{equation} T : H_{c_\sigma}^2(H/L, Z)\rightarrow H_{c_\tau}^2(G/K,W) . \end{equation}
Holographic operators, have properties quite similar to the symmetry breaking operators, and they satisfy a Proposition quite    similar to Proposition~\ref{prop:propertiesksc}. To follow, we state without proof (the proof is quite similar to the proof of facts valid for symmetry breaking operators)   facts about holographic operators necessary for further developments. A holographic operator $T$ is represented by a smooth kernel $K_T^c :H/L \times G/K \rightarrow Hom_\C(Z,W)$ so that for $z\in Z, h\in H$, $K_T(hL,\cdot )z \in H_{c_\tau}^2(G/K,W)$;   for $w\in W, s\in G$, $K_T( \cdot, sK )^\star w \in H_{c_\sigma}^2(H/L,Z)$  and  for $g_1 \in H_{c_\sigma}^2(H/L,Z), w\in W, x  \in G$, we have   $(T(g_1)(xK),w)_W =(g_1, K_T^c(\cdot, xK)^\star w)_{L_{c_\sigma}^2(H/L,Z)}$. The analogue to Proposition~\ref{prop:propertiesksc} is the following:
 \begin{prop} \label{prop:propertiesholo} Under the assumptions of the previous paragraph,   the following facts hold for an intertwining continuous linear map  $$T:  H_{c_\sigma}^2(H/L,Z) \rightarrow  H_{c_\tau}^2(G/K,W).$$
   Notation: $h,y \in H, x\in G, z\in Z, w\in W, g\in H_{c_\sigma}^2(H/L,Z).$
 \begin{enumerate}
 \item $ T(g)(xK)$ \\ \phantom{xxxxx} $=\int_{H/L} K_T^c(hL,xK)\big(c_\sigma(h,e)c_\sigma(h,e)^\star \big)^{-1}g(hL) dm_{H/L}(hL).$
 \item $  K_T^c(hL,xK)z=K_{T^\star}^c(xK,hL)^\star z=T (K_\sigma^c(\cdot,hL)^\star z)(xK).$
 \item $  K_T^c(hyL,hxK)=c_\tau (h,xK)K_T^c (yL,xK) c_\sigma (h,yL)^\star.$
 \item $   K_{T^\star}^c( xK, hL) w= K_{T}^c(hL, xK)^\star w=T^\star(K_\tau^c(\cdot,xK)^\star w)(hL).$
 \item $   \text{ The\, linear\, map,} \\  \phantom{xxx}   Z\ni z \mapsto K_T^c(eL,\cdot )  z =K_{T^\star}( \cdot, eL)^\star z \in H_c^2(G/K,W),$  is  a \, L-map.    Moreover, for each  $z\in Z , \\ \phantom{} K_T^c( \cdot , eL )  z =K_{T^\star}(\cdot, eL )^\star z \in H_{c_\tau}^2(G/K,W)[H_{c_\sigma}^2(H/L,Z)][Z]$.
 \item $   (T^\star (K_\tau^c(\cdot,xK)^\star w)(hL),z)_Z=(w,T (K_\sigma^c(\cdot,hL)^\star  z)(xK))_W   .$
 \item $ \forall g, \,  ( T (g)(xK), w)_{W}=(g,T^\star (K_\tau^c(\cdot,xK)^\star w)_{L_c^2(H/L,Z)} .$
     \end{enumerate}
 \end{prop}

 \medskip

  The   diagram below shows how to transfer intertwining operators $T$   on  group model of representations   to operators $T^c$ in the symmetric space  model. The corresponding relation between the kernels $K_T$, $K_{T^c}^c$ is \begin{equation}\label{eq:ttc} K_{T^c}^c(hL,xK)=c_\tau (x,e)K_T (h,x) c_\sigma (h,e)^\star, \, \forall h \in H, x \in G.\end{equation}    \cite[Proposition 3.7]{OV2} yields $K_{T^c}^c$ is a smooth function.

 \xymatrix{ \phantom{xx} & {H^2(G,\tau)} \ar[r]^{\cong\,\,\,}  & {H_{c_\tau}^2(G/K,W) }   & f \ar[r]^{E_{c_\tau}\,\,\,\,\,\,\,\,} & c_\tau(\cdot,e)f(\cdot)  \\
    \phantom{xx} & {H^2(H,\sigma)} \ar[u]_{T}   \ar[r]^{\cong\,\,\,} & {H_{c_\sigma}^2(H/L,Z)} \ar[u]^{T^c} & g \ar[r]^{E_{c_\sigma}\,\,\,\,\,\,\,\,} & c_\sigma(\cdot,e)g(\cdot). }

 \subsection{Duality Theorem in $G$-bundle space model} \label{sub:DualityinG} In this subsection we assume $(G,H)$ is a  symmetric pair. More precisely,  $G$ is a connected reductive group and $H$ is the identity connected component of the fixed points of an involution $\sigma$ of $G$.
 Let $K:=G^\theta$ a maximal compact subgroup for $G$ so that $\sigma$ commutes with $\theta$. Hence $L:=H\cap K$ is a maximal compact subgroup of $H$.  Let $H_0:=(G^{\sigma \theta})_0$ the associated subgroup to $H$, then $L$ is a maximal compact subgroup of $H_0$. Next, we fix $H^2(G,\tau)$ a Discrete Series representation for $G$ so that is $H$-admissible. Let $\mathcal U(\h_0)W:=\mathcal U(\h_0)\big( H^2(G,\tau)[W]\big)$ denote the subspace of $H^2(G,\tau)$ spanned by the left translates by $\mathcal U(\h_0)$ of the lowest $K$-type $H^2(G,\tau)[W]$ of $H^2(G,\tau)$. Finally, we fix a Discrete Series $H^2(H,\sigma)$ for $H$. Then,  the duality Theorem \cite{OV3} claims that \begin{equation*} Hom_H(H^2(H,\sigma),  H^2(G,\tau)) \equiv Hom_L(Z, \mathcal U(\h_0)W). \end{equation*} To follow, we describe the isomorphism $\equiv$ and some structural facts. Later on, we will present a analogue statement in the realm of symmetric space model realization of Discrete Series representations.
 In Section~\ref{sec:duality} we will present an independent proof of the Duality Theorem under the extra hypothesis: the inclusion $H/L \rightarrow G/K$ is holomorphic and   $H^2(G,\tau)$ is a holomorphic Discrete Series representation.

 The following statements are verified in \cite[Proposition 4.9]{OV3}.

 a) We write $res_L(\tau)$ as a sum of irreducible representations. That is, we write $res_L(\tau)=\sigma_1 \oplus \cdots \oplus \sigma_R$ with $(\sigma_j, W_j)$ an irreducible representation of $L$. Then,   each $\sigma_j$ is the lowest $L$-type of      a  Discrete Series $H^2(H_0,\sigma_j)$  for $H_0$,  and,  $\mathcal U(\h_0)W$ is $(\h_0, L)$-equivalent to the direct  sum over $1\leq j \leq R$ of the underlying Harish-Chandra modules for $H^2(H_0, \sigma_j)$. Formally, \begin{equation*} \mathcal U(\h_0)W \equiv \oplus_{1\leq j\leq R}\, H^2(H_0, \sigma_j)_{L-fin} \end{equation*}

 b) Let $\mathcal L_{W,H}$ the linear span of subspaces of $H^2(G,\tau)$ that realize the lowest $L$-type of each irreducible $H$-factor of $H^2(G,\tau)$. That is, \begin{equation*} \mathcal L_{W,H}=\oplus_{H^2(H,\mu) \in Spec(res_H(H^2(G,\tau))} H^2(G,\tau)[H^2(H,\mu)][Z_\mu] \end{equation*} Here, $(\mu, Z_\mu)$ is the lowest $L$-type of $H^2(H,\mu)$. Then, there exists
\begin{equation} \label{eq:exD} \text{  \, a \,bijective \,linear}\,  L-\text{map}\,\,\,  D \,\, \text{from}\,\,\, \mathcal L_{W,H} \,\text{ onto\,} \,  \mathcal U(\h_0)W .
\end{equation}

 c) Each  $T \in H^2(H,\mu)\rightarrow H^2(G,\tau)$ is represented by smooth kernel $K_T :H\times G \rightarrow Hom_\C(Z_\mu, W)$. Then, for each $z\in Z_\mu$, $K_T(e,\cdot)z \in H^2(G,\tau)[H^2(H,\mu)][Z_\mu]$, whence we obtain a map from $Z_\mu$ into $\mathcal U(\h_0)W$ via $D$.
 This is the map \begin{equation*} Z_\mu \ni z \mapsto D\big(G\ni x \mapsto K_T(e,x)z)\big)(\cdot)\in \mathcal U(\h_0)W. \end{equation*} The resulting map belongs to $Hom_L (Z_\mu, \mathcal U(\h_0)W)$. In \cite{OV3} it is shown that the map $T\mapsto \big(z \mapsto D(K_T(e,\cdot)z)(\cdot)\big )$ is bijective.

 d) For each $j$ we may and will realize $Z_{\mu_j}$ as a linear subspace of $W$. Thus, we may write $W=\oplus_{1\leq j \leq R} Z_{\mu_j}$, and we have the inclusions $H^2(H_0, \sigma_j) \subset L^2( H_0 \times_{\sigma_j} Z_{\mu_j})$ and the equality $L^2(H_0\times_\tau W)= \oplus_{1\leq j \leq R} L^2(H_0 \times_{\sigma_j}) Z_{\mu_j}$. We define the subspace

  \phantom{xxxxxx} $\mathbf{H}^2(H_0,\tau):= \oplus_{1\leq j \leq R} H^2(H_0 ,\sigma_j ) \subset L^2(H_0 \times_\tau W). $

  Now, since $H^2(G,\tau)$ consists of smooth functions, let $r_0 :H^2(G,\tau) \rightarrow  C^\infty (H_0 \times_\tau W)$ denotes the restriction map. It is shown in \cite[Proposition 1]{OV3} that $r_0$ yields a $(\h_0,L)$-equivalence from $\mathcal U(\h_0)W$ onto $\mathbf{H}^2(H_0,\tau)_{L-fin}$.

  e) The previous comments provide another version of the Duality Theorem. This is, $Hom_H(H^2(H,\sigma), H^2(G,\tau))\equiv Hom_L(Z, \mathbf{H}^2(H_0,\tau))$ via the map
  \begin{multline*} Hom_H(H^2(H,\sigma), H^2(G,\tau)) \ni T \\ \mapsto \bigg(Z\ni z \mapsto r_0(D(K_T(e,\cdot)z))(\cdot) \in Hom_L(Z, \mathbf{H}^2(H_0,\tau))\bigg). \end{multline*}

 f) A expression for the inverse to the map defined in e) has been computed in \cite[Section 4]{OV3}. A formula is:
   \begin{equation*}  K_T(h,x)z      =  (D^{-1}[\int_{H_0} K_\tau (h_0,\cdot)      (C(r_0 (D (K_T(e,\cdot)z))(\cdot))(h_0)dh_0 ])(h^{-1}x).
  \end{equation*} Here, $C=(r_0 r_0^\star)^{-1} $ is a bijective endomorphism  of $\mathbf{H}^2(H_0,\tau)$.

  \smallskip
  When, $D$ can be chosen equal to the identity map, the formula reads  \begin{equation*}  K_T(h,x)z      =   \int_{H_0} K_\tau (h_0,h^{-1}x)z)      (C (r_0 ( (K_T(e,\cdot))(\cdot))(h_0))dh_0 .
  \end{equation*}

  \begin{rmk} In the holomorphic setting, in the formulas in e), f) we are able to replace the linear map $r_0D$  by $r_0$. \end{rmk}

 \subsection{Duality Theorem in Symmetric space model} \label{sub:dualinsymm} The statement in this case is:

 $Hom_H(H_{c_\sigma}^2(H/L,Z), H_{c_\tau}^2(G/K, W))\equiv Hom_L(Z, \mathbf{H}_{c_\tau}^2(H_0/L, W))$

   A explicit isomorphism is: \begin{multline*} Hom_H(H_{c_\sigma}^2(H/L, Z), H_{c_\tau}^2(G/K, W)) \ni T^c \\ \mapsto \bigg(Z\ni z \mapsto r_0^c(D^c(K_{T^c}^c(e,\cdot)z))(\cdot) \in Hom_L(Z, \mathbf{H}_{c_\tau}^2(H_0/L,W))\bigg). \end{multline*}
 Here, $res_L(\tau)=\oplus_j \sigma_j$, \,\,\, $\mathbf{H}_{c_\tau}^2(H_0/L,W):=\oplus_j H_{c_\tau}^2(H_0, \sigma_j)$, \\ \phantom{xx} $\mathcal U(\h_0)W:= \mathcal U(\h_0)\big (\mathbf{H}_{c_\tau}^2(H_0/L,W) [res_L(W)]\big) \subset \mathbf H_{c_\tau}^2(H_0/L,W)$, \\ $D$ as in \ref{eq:exD},  $D^c:=  E_{c_\tau}DE_{c_\tau}^{-1} : E_{c_\tau}(\mathcal L_{W,H})=:\mathcal L_{W,H}^c \rightarrow \mathcal U(\h_0)W  $.
 \subsection{Differential operators}
 \subsubsection{  Generalities on differential operators}\label{sub:gener1}

 In \cite[Proposition 6.1]{OV3}, we have shown in the group model $H^2(G,\tau), H^2(H,\sigma)$ the following result: $\mathcal L_{W,H} =\mathcal U(\h_0)W$ if and only if for every $(\sigma, Z)$ every symmetry breaking operator is represented by a normal derivative  differential operator. Actually, if follows from the proof of the previous statement that if $S : H^2(G,\tau) \rightarrow H^2(H,\sigma)$ is continuous $H$-map so that $\forall z \in Z, K_{S^\star}( e, \cdot)
  z=K_S(\cdot,e)^\star z \in \mathcal U(\h_0)W$, then, $S$ is represented by a normal derivative differential operator, and conversely. For $X\in \mathcal U(\g)$, $R_X$ denotes infinitesimal right derivative by $X$. Then,
 our definition of normal derivative differential operator is based on the fact (for a proof \cite[Chap. II]{Varad}, \cite[chap V]{Wal}) A linear map $ D :\Gamma^\infty(G\times_\tau W) \rightarrow  \Gamma^\infty(G\times_\sigma Z)$ is a {\it differential operator} if and only if there exists finitely many smooth functions $c_\alpha :G\rightarrow Hom_\C(W,Z)$ so that $D(f)=\sum c_\alpha R_{X_1^{\alpha_1}\cdots X_n^{\alpha_n}}$ or a similar expression by means of left derivatives. Such a $D$ is invariant by left translations by $G$ if and only if, for every $\alpha$,  $c_\alpha$ is a constant function and $\sum_\alpha c_a X_1^{\alpha_1}\cdots X_n^{\alpha_n} \in (Hom_\C(W,Z)\otimes \mathcal U(\g))^L$.  As in \cite[4.0.1]{OV3}    $S : H^2(G,\tau) \rightarrow H^2(H,\sigma)$ is {\it represented by a  differential operator} (resp. {\it normal derivative differential operator}) if there exist a $G$-invariant differential operator   $D:   \Gamma^\infty(G\times_\tau W) \rightarrow     \Gamma^\infty(G\times_\sigma Z)$ so that $ \forall \,  f \in H^2(G,\tau)$ we have $  S(f)= res (D(f))$ (resp. if $D=\sum_\alpha c_\alpha R_{X_1^{\alpha_1}\cdots X_n^{\alpha_n}}, \text{with}\, X_1^{\alpha_1}\cdots X_n^{\alpha_n} \in \mathcal U(\h_0)$). Here, $res$ is the restriction map $ res: \Gamma^\infty(G\times_\sigma Z) \rightarrow \Gamma^\infty(H\times_\sigma Z)$. According to Theorem~\ref{prop:Sisnicediff}, in order to show that $S$ is the restriction of a differential operator, as in  Fact~\ref{eq:invdifop}, it suffices to show $S$ is the restriction of a global differential operator. We also recall that in \cite[Lemma 4.2]{OV2} it is shown that $S$ is the restriction of a differential operator if and only if $K_S(\cdot,e)^\star z$ is a $K$-finite vector for each $z\in Z$.

 \subsubsection{} \label{sub:gener2} Let $(\tau, W), (\eta,F)$ be two finite dimensional representations of $K$, $c_\tau, c_\eta$ cocycles from $G\times G/K $ into $Gl(W)$, $Gl(F)$. Then, from  $ C^\infty(G)\otimes W$ into $C^\infty(G)\otimes F$ we have two families of operators, namely the one's constructed by means of left infinitesimal translation $L_D, D \in \mathcal U(\g)\otimes Hom_\C(W,F)$, and the one's constructed via  right infinitesimal translation $R_D, D \in \mathcal U(\g)\otimes Hom_\C(W,F)$. We know \cite{Wal}, when $D \in (\mathcal U(\g)\otimes  Hom_\C(W,F))^{(Ad\otimes \tau^\vee \otimes \eta)(K) }$,  the operator $R_D$ transform the subspace $(C^\infty(G)\otimes W)^{(R_\cdot \otimes \tau)(K)}$ into the subspace $(C^\infty(G)\otimes F)^{(R_\cdot \otimes \eta)(K)}$, whence, $R_D$ defines a left invariant differential operator from $\Gamma^\infty(G\times_\tau W)$ into $\Gamma^\infty(G\times_\eta F)$.   In \cite{Wal} it is shown  that every $G$-left invariant differential operator from $\Gamma^\infty(G\times_\tau W)$ into $\Gamma^\infty(G\times_\eta F)$  is  obtained in this way.

 \smallskip

 Kobayashi-Pevzner have generalized the previous result, we present a version according to our future needs. We fix $H\subset G$ connected  reductive subgroups, $H$ closed in $G$. We fix respective maximal compact subgroups $L\subset H, K\subset G$ so that $L=K\cap H$. We fix $(\tau, W)$ (resp. $(\eta, F)$) a finite dimensional  representation of $K$ (resp. of $L$).
0295 Then, we have the natural inclusion $H/L \hookrightarrow G/K$ and the spaces $\Gamma^\infty(G\times_\tau W)\simeq C^\infty (G,W)^{(R\otimes W)(K)}$,  $\Gamma^\infty(H\times_\eta F)\simeq C^\infty (H,F)^{(R\otimes \sigma)(L)}$. In \cite{KP1}, they  have shown the space of $H$-invariant differential operators from $\Gamma^\infty(G\times_\tau W)$ into $\Gamma^\infty(H\times_\eta F)$ is isomorphic to the space $Hom_L(F^{\vee}, \mathcal U(\g) \otimes_{\mathcal U(\k)} W^{\vee} )$. A  isomorphism is achieved via the map $\phi \in Hom_L(F^{\vee}, \mathcal U(\g) \otimes_{\mathcal U(\k)} W^{\vee} )$ is mapped to the differential operator $D_\phi$, defined,  for $f\in C^\infty(G)\otimes W$,  by the equality,  \begin{equation} <D_\phi (f), z^\vee >=\sum_j <R_{u_j} (f), w_j^\vee >\vert_H  \, \text{ for} \,\, z^\vee \in F^\vee, \end{equation}  where $\phi (z^\vee)=\sum_j u_j w_j^\vee \, (u_j \in \mathcal U(\g), \, w_j^\vee \in W^\vee) $.

 \smallskip

 An equivalent expression for the $H$-invariant differential operator is
 \begin{fact}\label{eq:invdifop} \begin{equation*}  D_\phi(f)=\sum_j T_j R_{D_j}(f), \, T_j \in Hom_\C(W,Z), D_j \in \mathcal U(\g) \end{equation*}
  \end{fact} which corresponds to  $   \phi \in Hom_L(F^\vee , \mathcal U(\g)\otimes_{\mathcal U(\k)} W^\vee) $ defined by $\phi(z^\vee)=\sum_j D_j \otimes T_j^t(z^\vee)$. The $L$-equivariance of $\phi$ is represented by the equality \begin{equation} \sum_j D_j \otimes T_j = \sum_j Ad(l)(D_j ) \otimes \sigma(l) T_j \tau(l^{-1}), \,\, \forall \, l\in L \end{equation}
 \medskip

 We fix cocycles $c_\tau : G\times G/K \rightarrow Gl(W), c_\eta :H\times H/L \rightarrow Gl(F)$.
  The cocycles $c_\tau, c_\eta$ carries respective  parallelization  for the respective vector bundles $G\rightarrow G\times_\tau W$,  $H\rightarrow H\times_\eta F$  and respective isomorphisms  as the following picture shows

 \xymatrix{ \phantom{xx} & {\Gamma(G\times_\tau W)} \ar[r]^{\cong}  \ar@<1ex>[d]^S_{.}     & {\Gamma_{c_\tau} (G/K,W) } \ar@<1ex>[d]^{S^c} & f \ar[r]^{E_{c_\tau}\phantom{xxxxxx}} & c_\tau(\cdot,e)f(\cdot)  \\
    \phantom{xx} & {\Gamma(H \times_\sigma F)}  \ar@<1ex>[u]^T \ar[r]^{\cong} & {\Gamma_{c_\sigma} (H/L,F)} \ar@<1ex>[u]^{T^c}& g \ar[r]^{E_{c_\sigma}\phantom{xxxxxx} } & c_\sigma(\cdot,e)g(\cdot). }

    Here, $\Gamma, \Gamma_c$ represents   $\Gamma^\infty, \Gamma_{c_\tau}^\infty, L^2, L_{c_\tau}^2$ and so on.

 \medskip

 In particular, when $H=G$, $(\tau,W)=(\eta, F)$ we obtain three families of differential operators on $\Gamma^\infty(G/K,  W)$, namely,
 \begin{equation}\label{eq:difop} \text{For}, \, D \in \mathcal U(\g) \rightsquigarrow \dot{\pi}_{c_\tau} (D),\,\, L_D^c: = E_{c_\tau}  L_D E_{c_\tau}^{-1}. \end{equation} \begin{equation*} \text{For} \,\, D \in (\mathcal U(\g)\otimes Hom_\C(W,W))^{(Ad \otimes \tau^\vee \otimes \tau)(K)} \rightsquigarrow R_D^c := E_{c_\tau}  R_D E_{c_\tau}^{-1} . \end{equation*}   Moreover, the action of  $G$ on $G/K$  by left translation gives rise to another family of differential operators $L_D^{G/K} , D\in \mathcal U(\g) $ by the  formula for $X\in \g$, $L_X^{G/K}(f)(x):=\{\frac{d}{dt}  f(exp(-tX)x)\}_{t=0}$. We would like to recall the following identities \begin{multline*} \text{ For}\, x \in G, X \in \g, f \in C^\infty(G/K, W)= C^\infty(G)\otimes W) \\
 \dot{\pi}_{c_\tau}  (X)(f)(x)=L_X^{G/K}(f)(x)+ \big\{\frac{d}{dt}   c_\tau(exp(tX), exp(-tX)x)\big\}_{t=0} f(x).\end{multline*} \begin{equation}
 L_D^c=\dot{\pi}_{c_\tau} (D),\,\, D\in \mathcal U(\g). \end{equation}
 \begin{equation} L_D(f)(e)=R_{\widecheck D}(f)(e), D\in \mathcal U(\g), f\in C^\infty(G,W). \end{equation}

 Here, $D\mapsto \widecheck D$ is the anti automorphism of $\mathcal U(\g)$ associated to the map $\g \ni X \mapsto -X \in \g$.

 \smallskip
  The\, equality  \,$ L_D^c(f)(e)=R_{\widecheck D}^c(f)(e), D\in \mathcal U(\g), f\in C^\infty(G/K,W)$, is not true, as it readily follows from the formula for $L_X^c$.

 \smallskip
 However, it holds $[L_D^c (f)](e)= [R_{\widecheck D}( c_\tau (\cdot,o)^{-1}f(\cdot ))](e)$.

 \smallskip
 We could define $R_D^{cnv}$ for any $D\in \mathcal U(\g)\otimes Hom_\C(W,W)$ as follows: Let $P_0$ the projector of $C^\infty(G)\otimes W $ onto $(C^\infty(G)\otimes W)^{R\otimes \tau (K)}   $ , the projector is given by integration along $K$. For any $D\in \mathcal U(\g)$ we   set

 \phantom{xxxxxxxxxxxxxxxx} $R_D^{cnv} =E_{c_\tau} P_0 R_D E_{c_\tau}^{-1}   $.

 Certainly, the definition of differential operator in the language of supports, yields $R_D^{cnv}$ is a $G$-invariant differential operator on $\Gamma^\infty(G\times_\tau W)$. Whence, the result of Kobayashi-Pevzner just quoted, implies the existence of $D_0 \in (\mathcal U(\g)\otimes Hom_\C(W,W))^{(Ad\otimes \tau^\vee \otimes \tau)(K)}$ so that $R_D^{cnv} = R_{D_0}^c $. In Theorem~\ref{prop:Sisnicediff}  we show
   $D_0=Q_0(D)    $. Here, $Q_0$ is usual projector from $\mathcal U(\g)\otimes Hom_\C(W,W)$ onto $(\mathcal U(\g)\otimes Hom_\C(W,W))^{(Ad \otimes \tau^\vee \otimes \tau)(K)} $.
 \subsubsection{ } The technique developed in the next result has several   applications. We would like to point that the result generalizes \cite[Ch. II Theorem 4.6]{Hel}.  \begin{thm} \label{prop:Sisnicediff} Let $S: H^2 (G, \tau)\rightarrow H^2(H, \sigma) $ be a continuous $H$-intertwining map. We assume there exists a differential operator $D :C^\infty(G)\otimes W \rightarrow C^\infty (H)\otimes Z$ so that $D$ restricted to   $ H^2 (G, \tau)$ is equal to $S$ and $D=\sum_j T_j R_{D_j}, T_j \in Hom_\C(W,Z), D_j \in \mathcal U(\g)$. Then, there

 \noindent
   exists  $D_0=\sum_j P_j \otimes L_j  \in (\mathcal U(\g)\otimes Hom_\C(W,Z))^{(Ad \otimes \tau^\vee \otimes \sigma)(L)} $, so that $D_0 =\sum P_j R_{L_j} $ restricted to  $ H^2 (G, \tau)$ is equal to $S$. \end{thm} In Kobayashi-Pevzner   notation, $D_0 \in {\rm Diff}_H\big( (C^\infty(G)\otimes W)^{(Ad \otimes \tau)(K)} ,$ $ (C^\infty (H)\otimes Z)^{(Ad \otimes \sigma)(L)} \big)\equiv {\rm Diff}_H( \Gamma^\infty(G\times_\tau W) , \Gamma^\infty(H\times_\sigma Z)  )$   represents $S$. That is, in vector bundle language, and,  according to definition \ref{eq:invdifop},
 $S$ is equal to the restriction of a $H$-invariant differential operator  from $\Gamma^\infty(G\times_\tau W)    $ into $\Gamma^\infty(H\times_\sigma Z)   $.
 \smallskip

  \begin{proof}  For $\nu \in \widehat K$, let $\chi_\nu$ (resp. $d_\nu$) be the character of $\nu$ (resp. the dimension of $\nu$). Then, for the representation $R\otimes \tau$,   the isotypic component corresponding to $\nu$  is\\
 $(C^\infty(G)\otimes W )[\nu]:=\{ f: d_\nu \int_K \bar{\chi}_\nu(k) \tau(k) (f(xk)) dk =f(x) \}$, then $C^\infty(G)\otimes W=\mathrm{Cl}(\oplus_{\nu \in \widehat K}C^\infty(G)\otimes W[\nu])$ and $C^\infty(G)\otimes W[Triv]:=\{ f:   \tau(k)(f(xk) ) =f(x) \}$.
 \smallskip

 Similarly, for $\eta \in \widehat L$ and the action $R\otimes \sigma$    the isotypic component is
  $(C^\infty(H)\otimes Z )[\eta]:=\{ f: d_\eta \int_L \bar{\chi}_\eta(l) \sigma(l) (f(xl)) dl =f(x) \}$, then \phantom{sssssssssss}$C^\infty(H)\otimes Z=\mathrm{Cl}(\oplus_{\eta \in \widehat L}(C^\infty(H)\otimes Z)[\eta])$.

   \smallskip
  We recall   $Hom_L(Z^\vee, \mathcal U(\g)\otimes_{\mathcal U(\k)} W^\vee) = \big((\mathcal U(\g)\otimes_{\mathcal U(\k)} W^\vee)\otimes Z \big)[Triv] \equiv (S(\p)\otimes W^\vee \otimes Z)^{(Ad \otimes \tau^\vee \otimes \sigma)(L)}  $. The   equality is the expression for the isotypic component associated the trivial character for the representation   $Ad \otimes \tau^\vee \otimes \sigma$ of $L$. The   equivalence is given by means of  the symmetrization map.   For the representation   $Ad \otimes \tau^\vee \otimes \sigma$ of $L$,  we write the decomposition  $\mathcal U(\g)\otimes_{\mathcal U(\k)} W^\vee \otimes Z=\oplus_{\eta \in \widehat L} \,(\mathcal U(\g)\otimes_{\mathcal U(\k)} W^\vee \otimes Z)[\eta]$.

 \noindent
   We write $D=D_0 +\sum_{\eta \not= Triv \in \hat L} D_\eta, D_0 \in (\mathcal U(\g)\otimes_{\mathcal U(\k)} W^\vee \otimes Z)   ^{(Ad \otimes \tau^\vee \otimes \sigma)(L) }$, $D_\eta \in (\mathcal U(\g)\otimes_{\mathcal U(\k)} W^\vee \otimes Z)[\eta]$. We notice that $D\in (\mathcal U(\g)\otimes_{\mathcal U(\k)} W^\vee \otimes Z)[\eta]$ if and only if \begin{equation}\label{eq:proyDeta} \sum_s \int_L d_\eta \bar{\chi}_\eta(l)  \sigma (l)T_s \tau(l^{-1})R_{Ad(l)D_s} dl= \sum_s T_s R_{D_s} . \end{equation}

 \medskip

  To follow, we verify for $D=\sum_s D_s \otimes  T_s \in(\mathcal U(\g)\otimes_{\mathcal U(\k)} W^\vee \otimes Z)[\eta]$, it holds  $D (C^\infty (G)\otimes W)[Triv] )\subset (C^\infty(H)\otimes Z) [\eta]$.
  \smallskip
  For this, for $f \in (C^\infty (G)\otimes W)[Triv]$,  we must verify

  \phantom{xxxxxxxxxxxx} $\int_L d_\eta \bar{\chi}_\eta(l)  \sigma (l)D(f)(hl)dl=D(f)(h), h\in H$.
  \begin{equation*}
  \begin{split}
    \int_L d_\eta \bar{\chi}_\eta(l)  \sigma (l)D(f)(hl)dl &
  = \int_L d_\eta \bar{\chi}_\eta(l)  \sum_s \sigma (l)T_s (R_{ D_s}(f) (hl)) dl
  \end{split}
  \end{equation*}
    The equality    $(R_{ D_s}(f))(hl)=\tau (l^{-1}))(R_{Ad(l)D_s}(f))(h)$, applied to  the second member, yields the first member is equal to

   \phantom{xxxxxx} $= \int_L d_\eta \bar{\chi}_\eta(l)  \sum_s \sigma (l)T_s \tau(l^{-1}) (R_{Ad(l)D_s}(f))(h )dl= D(f)(h)$.

  \medskip
    We now finish the proof of the Theorem, for this we notice

  \phantom{xxxxxxxxxxxxxx}  $H^2(G,\tau)\subset C^\infty (G)\otimes W[Triv]$, and\\
  $S(H^2(G,\tau))=D(H^2(G,\tau))= D_0 (H^2(G,\tau)) \oplus \oplus_{\eta \not= Triv \in \widehat L} D_\eta (H^2(G,\tau))$.\\
    Since   our hypothesis is

     $S(H^2(G,\tau))\subset H^2(H,\sigma ) \subset C^\infty(H )\otimes Z[Triv]$, and the sum is direct we have $  D_\eta (H^2(G,\tau))=0$ for  $\eta \not= Triv \in \hat L $, so $D$ restricted to $ H^2(G,\tau)$     is equal to $ D_0 $ restricted to $H^2(G,\tau)$. \end{proof}
  Actually, the proof applies to the subspaces $(C^\infty (G)\otimes W)[Triv]$,  $(C^\infty(H )\otimes Z)[Triv]$, to obtain that  any $D=\sum_s T_s R_{D_s}$ so that \\ $D((C^\infty (G)\otimes W)[Triv])\subset (C^\infty(H )\otimes Z)[Triv]$,   can be replaced by the component $D_0 $ as above.

  \subsection{Differential symmetry breaking  operators}

  \medskip
  To follow, we apply  the previous definitions and results to the different    realization's   of Discrete Series representations.
  \subsubsection{Symmetry breaking op's in case $G$-bundle model}
  \medskip
   We  recall and sketch a proof of a result in \cite[Thm 4.3]{OV2}.
   \begin{fact}\label{fact:sdiffop} We assume the restriction to $H$  of $(L,H^2(G,\tau))$ is an admissible representation. Then,  any intertwining map $S: H^2(G,\tau)\rightarrow H^2(H,\sigma)$ is the restriction of a differential operator.\end{fact}
   To follow, we outline a verification of Fact~\ref{fact:sdiffop},   in turn, the proof yields how to represent $S$ as a differential operator from $K_S$ and viceversa. We recall in  \cite[Remark 3.8]{OV2}, we have shown that
      the function $ K_S(\cdot,e)^\star z$ is $L$-finite.
     Thus, owing to   \cite[Proposition 1.6]{Kob2}, "In an $H$-admissible  representation, $L$-finite vectors   are $K$-finite vectors",   our hypothesis forces that $K_S(\cdot,e)^\star z $ is a $K$-finite  vector in $H^2(G,\tau)$. We fix $0\not= w_0 \in W$.  Whence, the irreducibility of the representation $L_\cdot^G$ implies there exists a linear function $Z\ni z \rightarrow D_z \in \mathcal U(\g)$,  so that   $K_S(\cdot,e)^\star z=L_{D_z}^c K_\tau(\cdot,e)^\star w_0$.    Therefore,
  \begin{equation*}
     \begin{split} (S(f)(e),z)_Z & =\int_G (f(x), K_S(x,e)^\star z)_W dx \\
  & =\int_G (f(x), L_{D_z}(K_\tau) (x,e)^\star w_0)_W dx \\
  &= \int_G (L_{D_z^\star}f(x),  K_\tau (x,e)^\star w_0)_W dx\\
  &=(L_{D_z^\star}f(e),w_0)_W.
  \end{split}
  \end{equation*}
   Thus, $(S(f)(h),z)_Z=(S(L_{h^{-1}}(f))(e),z)_Z=(L_{D_z^\star}(L_{h^{-1}}(f))(e),w_0)_W$.
  \noindent
   We recall the equality $L_D (f)(e)=R_{\widecheck D} (f)(e)$. Hence,
    $(S(f)(h),z)_Z=(R_{\widecheck D_z^\star}(L_{h^{-1}} (f)) (e),w_0)_W =(R_{\widecheck D_z^\star}(f) (h),w_0)_W $. Whence, after we fix an orthonormal basis $\{z_j, 1\leq j \leq \dim Z \}$ for $Z$,  an expression of $S$ by means of left invariant differential operators is
   $$ S(f)(h)=\sum_{ 1\leq j \leq \dim Z} (R_{\widecheck D_{z_j}^\star}(f) (h),w_0)_W\,  z_j .$$ Next, we apply Theorem~\ref{prop:Sisnicediff} and obtain a differential operator $D_0$, according to Fact~\ref{eq:invdifop}, that represents $S$.  We define $T_j :W \rightarrow Z $ by the equality $T_j(w)=(w,w_0)_W \, z_j$. Hence, $S(f) =\sum_j T_j R_{\widecheck D_{z_j}^\star}(f) $. The component "$D_0$" of the right hand side, according to \ref{eq:proyDeta}, is $\int_L \sum_j  \sigma(l)T_j \tau (l^{-1}) R_{Ad(l)(\widecheck D_{z_j}^\star)} dl $. Next, we fix a ordered linear basis $\{ D_t\}_t$ for $\mathcal U(\g)$. Thus, $Ad(l)(\widecheck D_{z_j}^\star)=\sum_t a_{j,t}(l) D_t$ and we obtain
   $S(f)(h)=\sum_{j,t} \int_L \sigma(l)T_j \tau (l^{-1}) a_{j,t}(l)\, dl\, \, R_{D_t} (f)(h), $
    \noindent
   a expression of $S$ as differential operator in the realm of Fact~\ref{eq:invdifop}.
  \medskip
    A  expression for $S$ by means of infinitesimal left translations invariant differential operators is $$ S(f)(h)=\sum_{ 1\leq j \leq \dim Z} (L_{Ad(h)(D_{z_j}^\star) } (f)(h),w_0)_W\, z_j. $$ In the ordered linear basis $\{D_t\}_t$, we write $Ad(h)(D_{z_j}^\star) =\sum_t \phi_{t,j}(h) D_t$ with $\phi_{t,j}(\cdot)\in C^\infty (H)$, and we obtain the  expression $$ S(f)(h)=\sum_{t,j} \phi_{t,j}(h) (L_{D_t}f(h),w_0)_W \, z_j. \eqno{(\dag)}$$ This concludes the outline for the direct implication in  the proof of Fact~\ref{fact:sdiffop}

  Next, we answer the converse statement to Fact~\ref{fact:sdiffop}, that is, we sketch a way to recover the kernel $K_S$ when we know a representation  of $S$ as a differential operator. A tool  is  \cite[Proposition 16]{Kob2}  and the identity $S(K_\tau(\cdot,x)^\star w)(h)=K_S(x,h)w$, \cite[Proposition 3.7]{OV2}. The hypothesis yields the existence of  $D_j \in \mathcal U(\g)$, and $ \phi_j, \psi_j $,   $Hom_\C(W,Z)$-valued smooth functions, so that for $ f \in H^2(G,\tau)^\infty$, \begin{equation*}S(f)(h)=\sum_j \phi_j(h) L_{D_j}( f)(h)=\sum_j \psi_j(h) R_{D_j} (f)(h),   h \in H. \end{equation*} Both sums are finite sums!

   In \cite{OV2} it is verified that   $ K_\tau(\cdot,h)^\star w$ is a smooth vector for $H^2(G,\tau)$, hence, we have

  $ K_S(x,h)w =S(K_\tau(\cdot,x)^\star w)(h) $

    $=\sum_j \phi_j(h)L_{D_j}^{(1)}\big(K_\tau (\cdot ,x)^\star (w)\big)(h)= \sum_j \psi_j(h)R_{D_j}^{(1)}\big(K_\tau (\cdot,x)^\star (w)\big)(h). $

    Therefore,  knowing a expression of $S$ as differential operator,  we just follow the above recipe  to obtain the kernel $K_S$.

 \subsubsection[xx]{Symmetry breaking op's for symmetric space $G/K$-model}\label{sub:symmcase} We recall our hypothesis, the representation $res_H((\pi_{c_\tau}^c, H_{c_\tau}^2(G/K,W))$ is admissible and $S^c : H_{c_\tau}^2(G/K,W)\rightarrow H_{c_\tau}^2(H/L,Z)$ is symmetry breaking operator.   Under these hypothesis it follows $S^c$ is represented by a differential operator.

Note that we already have an isomorphism between the bundle model and the symmetric space model; but is convenient to work entirely within the symmetry space model.

 {\it To follow we show  how to compute a differential operator that represents   the operator $S^c$ knowing the kernel $K_{S^c}^c$   and viceversa.}

 We recall Proposition~\ref{prop:propertiesksc} shows \\ \phantom{xxxxx} $K_{S^c}^c(\cdot,eL)^\star z \in H_{c_\tau}^2(G/K,W) [H_{c_\sigma}^2(H/L,Z)][Z]$.

  We fix a nonzero $w_0 \in W$. Then, due that the representation $H_{c_\tau}^2(G/K,W)$ is $H$-admissible, we have $K_{S^c}^c(\cdot,eL)^\star z \in H_{c_\tau}^2(G/K,W)_{K-fin}$ \cite{Kob1}, also we have that infinitesimal representation in $H_{c_\tau}^2(G/K,W)_{K-fin}$ is algebraically irreducible. Thus, there exists a linear function $Z\ni z \mapsto D_z\in \mathcal U(\g)$ so that

 \phantom{equation} $\pi_{c_\tau}^c(D_z)\big(K_\tau^c(\cdot, e)^\star w_0\big)(\cdot)= K_{S^c}^c(\cdot,eL)^\star z, z\in Z.$

 Next, we point out equalities that will justify  a later development:

   $\pi_{c_\tau}^c(D_z)\big(K_\tau^c(\cdot, e)^\star w_0\big)(\cdot)=L_{D_z}^c \big(K_\tau^c(\cdot, e)^\star w_0\big)(\cdot)  $;

    $L_D^c(F)(x)=c_\tau (x,o)L_D^G  \big((c_\tau(\cdot,e))^{-1}F(\cdot)\big)(x)$; \,\,\, $c_\tau(e,yK)=I_W$;

  $K_\tau^c(xK,eK)=c_\tau(e,eK)K_\tau (x,e)c_\tau(x,eK)^{\star}= K_\tau (x,e)c_\tau(x,eK)^{\star}$.

    Equation~\ref{eq:kerofSc} justifies the first equality, the others are justified by means of the previous equalities
 \begin{equation*}
    \begin{split}   &  (S^c(f)(o),z)_Z\\ \phantom{xx}  &=\int_{G/K}((c_\tau(x,o)c_\tau(x,o)^\star)^{-1}f(x), K_{S^c}^c(x,eL)^\star z)_W dm_{G/K}(x) \\
 & =\int_{G/K} ((c_\tau(x,o)c_\tau(x,o)^\star)^{-1}f(x), L_{D_z}^c \big(K_\tau^c(\cdot, e)^\star w_0\big)(x) )_W dm_{G/K}(x)  \\
 \phantom{xx} & =\int_{G/K} (c_\tau(x,o)^\star (c_\tau(x,o)^\star)^{-1}c_\tau(x,o)^{-1}f(x),\\ \phantom{xx}  & \phantom{xxxxccccxxxxxxxxxccx} L_{D_z} \big(c_\tau(\cdot,e)^{-1}  K_\tau^c(\cdot, e)^\star w_0\big)(x) )_W dm_{G/K}(x)  \\
 &= \int_{G/K} (L_{D_z^\star}\big(c_\tau(\cdot,o)^{-1}f(\cdot)\big)(x),\,  c_\tau(x,o)^{-1} K_\tau^c (x,e)^\star w_0)_W dm_{G/K}(x)\\
 &= \int_{G/K} (L_{D_z^\star}\big(c_\tau(\cdot,o)^{-1}f(\cdot)\big)(x), \\ & \phantom{xxxxxxxxxxxxxxxxx} \,  c_\tau(x,o)^{-1}c_\tau(x,o) K_\tau  (x,e)^\star w_0)_W dm_{G/K}(x)\\
 &=(L_{D_z^\star}\big(c_\tau(\cdot,o)^{-1}f(\cdot)\big)(e),w_0)_W\\
 &=(R_{\widecheck D_z^\star}\big(c_\tau(\cdot,o)^{-1}f(\cdot)\big)(e),w_0)_W.
 \end{split}
 \end{equation*}
 We compute,
 \begin{equation*}
 \begin{split} (S^c(f)(h),z)_Z &=(c_\sigma(h,o)S^c(L_{h^{-1}}^c(f))(o),z)_Z\\
 & = ( S^c(L_{h^{-1}}^c(f))(o),c_\sigma(h,o)^\star z)_Z     \\ &=(R_{\widecheck D_{c_\sigma(h,o)^\star z}^\star}(L_{h^{-1}}(c_\tau(\cdot,o)^{-1}f(\cdot)))(o),w_0)_W \\ &= (R_{\widecheck D_{ c_\sigma(h,o)^\star z}^\star}( (c_\tau(\cdot,o)^{-1}f(\cdot)))(h),w_0)_W.
 \end{split}
 \end{equation*}
 We fix a linear basis $\{D_i\}$ for $\mathcal U(\g)$. Then, it readily follows from the fact, for $ X\in \g$,  $R_X$ is a derivation, that  there exists a finite family
    $\phi_i$  of   $Hom_\C(W,Z)$-valued  smooth functions on $H$, giving,
 \begin{equation}\label{eq:scasdiffop} S^c(f)(h\cdot o)= \sum_{1\leq j \leq \dim Z} (\sum_{i=1}^N \phi_i(h) R_{D_i} (f)(h\cdot o), z_j)_Z \, z_j. \end{equation}
   whence, $S^c$ is represented by a differential operator.

  Next, in a similar path to the development of the $G$-model,  we apply Theorem~\ref{prop:Sisnicediff} and obtain a differential operator $D_0$, according to Fact~\ref{eq:invdifop}, that represents $S^c$.

 \medskip

 To follow, we describe an algorithm to compute the kernel $K_{S^c}^c$ when we know a representation of $S^c$ by means of differential operators.

 The hypothesis is that $\pi_{c_\tau}$ is $H$-admissible, the tools are \ref{eq:scasdiffop} and the equality $ K_S^c(xK,hL)w=K_{S^\star}^c(hL,xK)^\star w=S (K_\tau^c(\cdot,xK)^\star w)(hL)$. Therefore

 \begin{equation*}
    \begin{split}  K_S^c(xK,hL)w   \\ &  =S (K_\tau^c(\cdot,xK)^\star w)(hL) \\ & =\sum_{1\leq j\leq \dim Z} (\sum_{i=1}^N \phi_i(h) R_{D_i} (K_\tau^c(\cdot,xK)^\star w)(h\cdot o), z_j)_Z \, z_j.
    \end{split}
    \end{equation*} and we have computed the kernel $K_{S^c}$.

   \begin{rmk} We would like to point out, the following fact shown in \cite[Lemma 4.2]{OV2}. $S^c$ is represented by a differential operator if and only if the function $K_{S^c}(\cdot,e)^\star z $ is a $K$-finite vector for each $z\in Z$.
   \end{rmk}

 \subsubsection{Case of  holomorphic imbedding $H/L \rightarrow G/K$}For this paragraph we assume, both symmetric spaces are Hermitian symmetric and the natural inclusion is holomorphic. We consider two holomorphic Discrete series representations realized, respectively on space of functions defined over the corresponding bounded symmetric domains obtained by Harish-Chandra. In this part we follow the notation in Section~\ref{sec:duality}. We quote the necessary notation.  $\g=\k +\p, \h=\l +\p_\h=\k \cap \h +\p \cap \h, \p=\p^+ +\p^-, \p_\h^+ =\p^+ \cap \h$, $\mathcal D \subset \p^+ $ the Harish-Chandra realization of the Hermitian symmetric space $G/K$,  $\mathcal D_\h \subset \p_\h^+ $ the realization of $H/L$. Thus, we have the holomorphic inclusions $\mathcal D_\h \subset \mathcal D $. Let $(\tau, W)$ (resp. $(\sigma ,Z) $) irreducible representations of $K$ (resp. $L$) so that the respective spaces of holomorphic functions $(L^\tau, V_\tau)$ (resp. $(L^\sigma, V_\sigma)$ on $L_\tau^2(\mathcal D , W)$ (resp. $L_\tau^2(\mathcal D_\h , Z)$) are nonzero. From now on, we assume $(L^\sigma, V_\sigma)$ is a irreducible $H$-subrepresentation of $(L^\tau, V_\tau)$. The action of $\g$ in $V_\tau$ has been computed by \cite{JV}, see \ref{sub:jv}, the formulas they obtained let us state.  For every $D\in \mathcal U(\g)$ (resp. $D\in \mathcal U(\h)$),

 $L_D^\tau$ (resp. $L_D^\sigma$) is a holomorphic differential operator on $V_\tau$ (resp. $V_\sigma$).

 \smallskip

 Next,  we fix a symmetry breaking operator $S : V_\tau \rightarrow V_\sigma$. Then, due that $V_\sigma$ is a reproducing kernel space, as in \ref{prop:propertiesksc}, it follows  there exists $K_S : \mathcal D \times \mathcal D_\h \rightarrow Hom_\C(W,Z)$ a smooth function, anti holomorphic in the first variable and holomorphic in the second variable, so that \begin{equation*} S(f)(z)=\int_\mathcal D K_S(w,z) K_\tau^c(\bar w,w)^{-1} f(w)\, dm_{G/K}(w),\ \forall\, f \in V_\tau,z\in \mathcal D_\h   . \end{equation*}
Here, we have appealed to equation~\ref{eq:ktaucktau} and its consequences.\\
  Our hypothesis implies $V_\tau$ is {\it $H$-admissible}, next, reproducing word by word the computation in \ref{sub:symmcase} and recalling $L_D^\tau$ is a holomorphic differential operator, we obtain $S$ is represented by a holomorphic differential operator.

  \cite{KP1} have strengthed this observation  to:

  \smallskip
  $S$ is represented by a constant coefficient holomorphic differential operator.

 \smallskip
 Actually, this follows from that $L_D^\tau, D\in \mathcal U(\p^+) \, (resp.\,  L_D^\sigma, D\in \mathcal U(\p_\h^+))$ is a constant coefficient holomorphic differential operator.

  \smallskip
  By means of the same computations as in \ref{sub:symmcase} we can derive $K_S$ from the knowledge of a differential operator that represents $S$ and viceversa. This problem has been analyzed by \cite[Section 3.3]{Na} in the holomorphic setting.

  \subsubsection{On the uniqueness of the map $z \rightarrow D_z$} In either the group model or symmetric space model, the function $(z \rightarrow D_z) \in Hom_\C (Z, \mathcal U(\g))$ so that $L_{D_z}^\cdot K_\tau^\cdot (\cdot,e)^\star w_0=K_S(\cdot,e)^\star z$ is far from unique, since  if we consider the annihilator $\mathcal A$ in $\mathcal U(\g)$ of the copy of $W$ in each model, then we may replace the function $D$ by $D+\tilde D$, where $\tilde D$ is an arbitrary element of $Hom_\C (Z, \mathcal A)$. However, when we consider the case $H/L \rightarrow G/K$ is a holomorphic embedding and both $V_\tau, V_\sigma$ are holomorphic representations,  we may replace the function $Z \ni  z\mapsto D_z\in \mathcal U(\g)$ by a function from $Z$ that takes on values  in the space $\mathcal U(\p^-)\otimes W$,   then, the new function $z\mapsto D_z$ is unique.

 \begin{prop} \label{prop:uniqness} $H/L \rightarrow G/K$ is a holomorphic embedding and both $V_\tau, V_\sigma$ are holomorphic representations. Then, there exists a unique linear function $Z\ni z\mapsto D_z:=\sum_i D_z^i \otimes w_z^i \in \mathcal U(\p^-)\otimes W$ so that $\forall z \in Z, \,K_S(\cdot,e)^\star z = \sum_i L_{D_z^i}^\tau (K_\tau(\cdot, e)^\star w_z^i)$.
 \end{prop}
  \begin{proof} Under our hypothesis, we may and  use   the notation in Section~\ref{sec:duality}.  Thus,    $(V_\tau)_{K-fin}$ is $K$-{\it isomorphic} to $\mathcal U(\p^-) \otimes W$ by means of the map $\mathcal U(\p^-) \otimes W \ni D\otimes w \mapsto L_D^\tau (K_\tau(\cdot, e)^\star w)(\cdot)$,  hence,  after we fix a ordered linear basis $\{w_i, i=1\cdots \dim W\}$ for $W$,  for each $v \in (V_\tau)_{K-fin}$, there exists unique family $ (D_i)_{1\leq i \leq \dim W} \in \mathcal U(\p^-)  $ so that $v=\sum_i L_{D_i}^\tau K_\tau (e,\cdot)^\star w_i$. Therefore, for each $z\in Z$ there is  unique family $D_z^i \in \mathcal U(\p^-) $ so that $K_S(\cdot,e)^\star z = \sum_i L_{D_z^i}^\tau (K_\tau(\cdot, e)^\star w_i)$. We define the new function $D_z$ to be function $Z\ni z \mapsto \sum_i D_z^i \otimes w_i$. This function does not depend on the choice of basis and is unique.
 \end{proof}

   A consequence of the previous discussion is that whenever $\tau $ is   one dimensional, that is, $V_\tau$ is a {\it scalar representation}, there is {\bf unique}, up to  a constant that  depends only  on the choice of $w_0$,  map $Z\ni z \rightarrow D_z \in \mathcal U(\p^-)$ so that $K_S(\cdot,e)^\star z =  L_{D_z }^\tau (K_\tau(\cdot, e)^\star w_0)$.

 \begin{rmk} We like to point out, that Nakahama in \cite[Section 3.3]{Na} has consider the problem of representing intertwining operators via differential operators  under the hypothesis  of Proposition~\ref{prop:uniqness}. His solution is  constructive and quite different to the one we found.
 \end{rmk}

 \subsection{Some properties of symmetry breaking operators} Here we treat a general discrete
series representation  (and not necessarily a symmetric pair for the branching
problem); we recall the Schmid construction via an elliptic operator (the "Schmid operator"),
namely the kernel of this operator yields the Discrete Series in question. Under the assumption $res_H( (L_\cdot, H^2(G,\tau)))$ is $H$-admissible in \cite[Theorem 4.11]{OV2} we find a proof that any continuous  symmetry breaking operator from $H^2(G,\tau)$ into $H^2(H,\sigma)$ extends to continuous symmetry  breaking operator from the maximal model of first representation given by the kernel in $C^\infty(G\times_\tau W)$  of the Schmid operator, into the maximal model of the second representation given by the Schmid operator,  and that any continuous symmetry breaking operator from the maximal model of   $(L_\cdot, H^2(G,\tau))$ into the maximal model of $(L_\cdot, H^2(H,\sigma))$ is represented by a differential operator. In \cite[Theorem 5.13]{KP1}, under the assumption of $res_H((L_\cdot, H^2(H,\sigma)) $ being $H$-admissible, they show, in the realm of holomorphic Discrete Series, that any symmetry breaking continuous operator from the maximal model given by holomorphic functions on $\mathcal D_G$ with values into the lowest $K$-type, into the maximal model given by holomorphic functions on $\mathcal D_H$ into the lowest $L$-type, carries the unitary model inside the maximal model  onto the unitary model contained in the maximal model. We would like to point out, that the result, as well as the proof, yields the following generalization.

 Let $Ker(D_G)$   denote the kernel of the Schmid operator, \cite{OV2},  owing to $D_G$ is elliptic, $Ker(D_G)$ is contained in $\Gamma^\infty(G\times_\tau W)$. As usual, we endow $\Gamma^\infty(G\times_\tau W)$ with the smooth topology. That is, the topology of uniform convergence on compact subsets of the sequence as well as any of its derivatives.   We have: $H^2(G,\tau)\subset Ker(D_G)$ is a dense subspace, the subspace of $K$-finite vectors of each of them are identical. A similar results holds for $H$. Let us recall that when $H^2(G,\tau)$ is a holomorphic Discrete Series, the Schmid operator is equal to the $\bar{\partial}$-operator.
 \begin{thm}\label{thm:xx} We assume $res_H((L_\cdot, H^2(G,\tau))$ is a  $H$-admissible representation. Then, the following four statements hold:

 (a) Any continuous $H$-intertwining linear map from $Ker(D_G)$ into $Ker(D_H)$ is the restriction of a differential operator.

 (b) Any continuous $H$-intertwining map from $H^2(G,\tau)$ into $H^2(H,\sigma)$ extends to a continuous linear map from $Ker(D_G)$ into $Ker(D_H)$.

 (c) Any nonzero continuous $H$-intertwining linear map from $Ker(D_G)$ into $Ker(D_H)$ maps continuously $H^2(G,\tau)$ onto $H^2(H,\sigma)$.

 $(d)$ Any $(\h,L)$ morphism   from $H^2(G,\tau)_{K-fin}$ into $H^2(H,\sigma)_{L-fin}$ extends to a continuous intertwining map from $H^2(G,\tau)$ into $H^2(H,\sigma)$.
 \end{thm}
 In \cite[Theorem 3.6]{Na}, for the holomorphic setting,  we find a proof of some of the statements in the Theorem.
 \begin{proof}[Proof of Theorem~\ref{thm:xx}]
 {\it (a)} and {\it (b)} are shown in \cite{OV2}. For {\it (c)} we follow \cite[Theorem 5.13]{KP1} quite closely.  Let $S : Ker(D_G)\rightarrow Ker(D_H)$ a continuous $H$-map.  $Ker(D_G)$ is maximal model \cite{Sch} \cite{Wo}, yields

 \noindent
 $S(Ker(D_G)_{K-fin}) =S(H^2(G,\tau)_{K-fin})$

 \phantom{xxxxxxxxxxxxxxxxxxxxxxxx} $= Ker(D_H)_{L-fin}=H^2(H,\sigma)_{L-fin}$.

 The hypothesis $res_H((L_\cdot, H^2(G,\tau))$ is $H$-admissible, gives \cite{DV} \\ $res_L((L_\cdot, H^2(G,\tau))$ is $L$-admissible and the work of  Kobayashi \cite{Kob1} \cite{Kob2}, we write, for short, $V_\lambda^G:=H^2(G,\tau),  V_{\mu_j}^H:=H^2(H,\sigma_j)$, then,

 $H^2(G,\tau)=\oplus_{1\leq j< \infty} V_\lambda^G[V_{\mu_j}^H]=\oplus_{1\leq j< \infty} m_j V_{\mu_j}^H$ \, (Hilbert sum),

 $H^2(G,\tau)_{K-fin}=   \oplus_{1\leq j<\infty} [m_j (V_{\mu_j}^H)]_{L-fin}$ \, (algebraic sum).

 Here, $\forall j, 1\leq m_j <\infty $; for $ i\not= j, V_{\mu_i}^H \not\equiv V_{\mu_j}^H$ .

 Thus, for all $j$ but one, $S$ carries the $V_{\mu_j}$ isotypic component in $H^2(G,\tau)_{K-fin}$  to the zero subspace, and let's say for $\mu_1=\mu$, ($V_\mu^H=H^2(H,\sigma)$),  $S(m_1(V_{\mu_1}^H)_{L-fin})\not= \{0\}$. Let $B_j$ linear, invariant $\mathcal U(\h)$-irreducible subspaces so that we  write $(V_\lambda)_{K-fin} [V_\mu^H]=m_1(V_{\mu_1}^H)_{L-fin}=\oplus_{1\leq j \leq m_1} B_j$,  and let's say $S(B_j)\not= \{0\} $ for exactly  $1\leq j \leq m \leq m_1$. Then $S$ restricted to each $B_j, 1\leq j \leq m$ is a equivalence with $H^2(H,\sigma)_{L-fin}$. By classical theory of Harish-Chandra, \cite[Lemma 8.6.7]{Wal}, $S$ extends to a continuous linear  $H$-bijection $S_j$  from the closure of $B_j$ in $H^2(G,\tau)$ onto $H^2(H,\sigma)$. We define $\tilde S=\sum_{1\leq j \leq m} S_j $ on $\oplus \mathrm{Cl}(B_j)_{1\leq j \leq m}$ and $\tilde S=0$ in either   $\mathrm{Cl}(\sum_{2\leq j <\infty}  m_j \mathrm{Cl}(V_{\mu_j}^H)_{L-fin}))$ or $\oplus_{m+1\leq j \leq m_1} \mathrm{Cl}(B_j)$. Thus, $\tilde S$ is a continuous $H$-map from $H^2(G,\tau)$ onto $H^2(H,\sigma)$. Owing to $(b)$, the map $\tilde S$ extends to a continuous linear map from $Ker(D_G)$ into $Ker(D_H)$ and the map $\tilde S$ restricted to $H^2(G,\tau)_{K-fin}$ agrees with the restriction of $S$. Since the subspace of $K$-finite vectors is dense in $Ker(D_G)$, we have they agree everywhere. Thus, $S$ maps $H^2(G,\tau)$ onto $H^2(H,\sigma)$, and we have shown $ (c) $. Finally, $(d)$ is shown in \cite{Kob4}.  \end{proof}

 \begin{rmk} {\it Summary on automatic continuity } After, we assume \\ $res_H((L_\cdot, H^2(G,\tau))$ is a admissible representation. Then, in \cite{Kob4} we find a proof of the following statements:

 $a)$ Any $(\h,L)$ morphism $T$ from $H^2(G,\tau)_{K-fin}$ into $H^2(H,\sigma)_{L-fin}$ extends to a continuous intertwining map from $H^2(G,\tau)$ into $H^2(H,\sigma)$.

 $(b)$ Any continuous intertwining linear operator from the space of smooth vectors in $H^2(G,\tau)$ into the space of smooth vectors in $H^2(H,\sigma)$ extends to a continuous morphism between the corresponding maximal model representations.

 We do not know whether or not $a), b)$ is true after we drop the $H$-admissibility assumption.

 \end{rmk}

 \section{Symmetry breaking operators as generalized gradients.}

 In \cite{DES} and references therein, we find a proof that  intertwining operators between representations realized as spaces of holomorphic functions are given by covariant differential operators. We may reinterpret their results in the language of symmetric spaces as follows,  consider   the symmetric pair $(G,G)$ so that $G/K\rightarrow G/K$ is a holomorphic embedding, and   respective  holomorphic representation's of  $G,G$.  Then,  they have shown that each symmetry breaking operator is represented by a covariant differential operator. In this section we generalize their result to an arbitrary symmetric pair $(G,H)$ and a $H$-admissible Discrete Series representation for $G$. To begin with, we recall that in \cite{OV2} we find a proof  that existence of a nonzero $H$-symmetry breaking operator represented by a differential operator forces the Discrete Series of $G$ that we are dealing with is $H$-discretely decomposable. In the same paper we find a proof that if every $H$-symmetry breaking operator is represented by differential operators, then the representation is $H$-admissible.

 For this section, $G$ is a arbitrary connected semisimple Lie group $K=G^\theta $ is a maximal compact subgroup of $G$. $(G,H:=(G^{\sigma})_0) $ is a symmetric pair, $H_0:=(G^{\sigma \theta})_0$ the associated subgroup to $H$.

 We fix $(L_\cdot , H^2(G, \tau))$   a $H$-admissible Discrete Series for $H$ of   lowest $K$-type $(\tau,W)$.

 Let $(L_\cdot, H^2(H, \sigma))$ be  a irreducible factor of $res_H(\pi)$.  The lowest $L$-type of $H^2(H, \sigma) $ is $(\sigma, Z)$. We note that for each representative of the equivalence class associated to $(\sigma, Z)$, we obtain a different space of functions $H^2(H,\sigma)$. We show that the expression of each symmetry breaking operator as "generalized gradient operator" is obtained from a particular choice of a representative of the class of $(\sigma, Z)$. In order to present the definition of "generalized gradient operator" we recall definitions and facts from \cite{OV}. For this section $\bullet$ denotes the natural representation of $L$ in the vector space  $ S^n(\p_{\h_0})\otimes W$ .The {\it normal derivative operator} of order $n$, $r_n$, is a map
   from \\ $ H^2(G,\tau) $ into $ C^\infty(H\times_{\bullet} Hom_\C(S^n(\p_{\h_0}), W))\equiv  C^\infty(H\times_{\bullet} S^n(\p_{\h_0})\otimes W) $.

  The formal definition of $r_n$ is:   for $X_1,\cdots, X_n \in \p_{\h_0}$,

  \phantom{xxxxxxx} $r_n(f)(X_1 \cdots X_n)(h)=\sum_{\iota \in \mathfrak S_n}R_{X_{\iota (1)}} \cdots R_{X_{\iota  (n)}} (f)(h)$.

   The first order normal derivative map is a gradient in the direction orthogonal to the tangent space to $H/L \hookrightarrow G/K$. The $n^{th}$-order normal derivative may be thought as    a  iteration of   first order normal derivative, whence, we call a $n^{th}$-order normal derivative a "generalized gradient", actually for $H=\{e\}$, $r_n$ is the iteration of   gradients. In  \cite{OV} we find a proof for: $r_n(H^2(G,\tau)) \subset L^2(H\times_{\bullet} Hom_\C(S^n(\p_{\h_0}), W))$ and $r_n$ is (2,2)-continuous.

 A {\it generalized gradient} representation of a symmetry breaking operator $S: H^2(G,\tau)\rightarrow H^2(H,\sigma)$ is a composition $S=\tilde R r_n$, where $R:  S^n(\p_{\h_0})\otimes W \rightarrow Z  $   is a $L$-map, and,   $\tilde R: C^\infty(H\times_{\bullet} S^n(\p_{\h_0})\otimes W) \rightarrow C^\infty(H\times_{\sigma} Z) $ is the map $\tilde R(g)(h):=R(g(h)), g \in C^\infty(H\times_{\bullet} S^n(\p_{\h_0})\otimes W), h\in H $. We would like to point out that the uniqueness of the lowest $L$-type of a Discrete Series representation yields that  any nonzero intertwining operator from $H^2(H,\sigma)\subset L^2(H\times_\sigma Z)$ into $H^2(H,\sigma_0)\subset L^2(H\times_{\sigma_0} Z_0)$ is the restriction of a operator of the type $\tilde R$, where $R$ is a equivalence from $Z$ onto $Z_0$.

 \begin{prop} \label{prop:des} Under the stated hypothesis,
 let $S : H^2(G,\tau) \rightarrow H^2(H,\sigma)$ be a symmetry breaking operator, then for a convenient choice of a representative of the equivalence class of $(\sigma, Z)$, $S$ is represented as a generalized gradient operator.
 \end{prop}
 \begin{proof} We define $T:=S^\star$. In \cite[3.1.7]{OV3}, we find a proof for the inclusion

 \phantom{xxxxxxxxx}  $T(H^2(H,\sigma)[Z])\subseteq  H^2(G,\tau)[H^2(H,\sigma)][Z]$.

 We claim  the subspace $T(H^2(H,\sigma)[Z])$ is intrinsic, that is, does not depend on the choice of  representative of the equivalence class of $(\sigma, Z)$.  In fact,    we fix an intertwining linear operator, which we may assume is unitary,  $R $  from $ (\sigma, Z) $ onto $ (\sigma_1, Z_1) $. Hence, the map  $ L^2(H\times_\sigma Z) \ni g \stackrel{\tilde R}\mapsto \big(h\mapsto R(g(h))\big)\in L^2(H\times_{\sigma_1} Z_1) $ is a unitary, $H$-equivalence that maps $H^2(H,\sigma)$ onto $H^2(H,\sigma_1)$, as well as $\tilde R$  maps $H^2(H,\sigma) [Z]$ onto $H^2(H,\sigma_1)][Z]$. Finally, $   T \tilde R^{-1} $ is equivalent to $T$.  Whence, we have shown the claim.

 Next, we fix the $n\geq 0$ so that for $j<n$,

 \phantom{xxxx}$r_j(T(H^2(H,\sigma)[Z]))=\{0\}, \text{and}, \,   r_n(T(H^2(H,\sigma)[Z]))\not=\{0\}$.

 To follow, for simplicity of exposition, we assume the rank of $K$ is equal to the rank of $L$ and the system of positive roots $\Psi$  determined by the Harish-Chandra parameter of $H^2(G,\tau)$ has the Borel de Siebenthal property,  for unexplained notation or  facts we refer to \cite[Section 4]{OV3}. The duality Theorem~\ref{sub:DualityinG}, \cite[Theorem 3.1]{OV3} together with the elements in the  proof of \cite[Lemma 4.5]{OV3} implies the Harish-Chandra   parameter of $(\sigma, Z)$ is equal to $\lambda +\rho_n +B$, here,  $B$ stands for a sum of noncompact  roots in $\Psi_{\h_0}$, we observe that  the coefficient, of the unique noncompact simple $\beta$ for $\Psi_{\h_0}$ in the expression of $B$ as a linear combination of simple roots,    is equal to $n$.

 An argument based on the fact, that any noncompact root in $\Psi_{\h_0}$ is equal to $\beta$ plus a linear combination of compact simple roots, as well as that the highest weight of any irreducible component of a  tensor product is equal to the highest weight of one factor plus  weight of the other factor,  shows that  $\lambda +\rho_n -\rho_L +B$ is not a $L$-type for $S^j(\p_{\h_0})\otimes W$ and $j<n$.

 We note that since $r_n(T(H^2(H,\sigma)))\not= 0$, Frobenius reciprocity yields:

 \phantom{ssssss} $r_n(T(H^2(H,\sigma)))$ contains a $L$-type in $S^n(\p_{\h_0})\otimes W$

 Now, a result of Schmid, see  \cite[Proof of Lemma 4.5]{OV3} gives that a $L$-type of $H^2(H,\sigma)$ is equal $\lambda +\rho_n -\rho_L +B+B_h$ where $B_h$ stands for a sum of roots in $\Psi_{\h}^n$. Next,  comparing  the coefficient of the  noncompact the simple root of $\Psi$ in $\lambda +\rho_n -\rho_L +B+B_h$ and the weight structure of $S^n(\p_{\h_0})\otimes W$ we obtain $B_h=0$ and hence, actually  $Z$ is an $L$-type of  $S^n(\p_{\h_0})\otimes W$. Thus, $H^2(H,\sigma)$ does not occur as a subrepresentation for $L^2(H\times_\bullet Hom_\C(S^j(\p_{\h_0}), W))$ for $j <n$.

  We set  $Z_0:=r_n(T(H^2(H,\sigma)[Z])))$ and we conclude the proof of Proposition~\ref{prop:des}. The subspace $Z_0$  of $S^n(\p_{\h_0})\otimes W$ affords a representation of $L$ equivalent to $(\sigma, Z)$. Let $P$ denote the orthogonal   projector in $S^n(\p_{\h_0})\otimes W$ onto $Z_0$. Then, $g \mapsto \tilde P(g)(\cdot):=P(g(\cdot))$ maps $ L^2(H\times_{\bullet} Hom_\C(S^n(\p_{\h_0}), W))$ into $L^2(H \times_\bullet Z_0)$, and $R:=\tilde Pr_n T$ is   a nonzero $H$-intertwining continuous linear map   from $H^2(H,\sigma)$ onto $H^2(H,\bullet)$. Hence, $R=\tilde R_0$, with $R_0 $ a $L$-map from $Z$ onto $Z_0$,  and the equality  $\tilde R_0=  \tilde Pr_n S^\star$ is true. Now $S S^\star =T^\star T= c I_{H^2(H,\sigma)}$ $c>0$. Thus, $S=\tilde R_0^{-1}\tilde P r_n$. Whence, we have verified   Proposition~\ref{prop:des}. \end{proof}

 \section{Holomorphic Embedding,  Duality Theorem.}\label{sec:duality} In this section we present a new and self contained proof of the Duality theorem for the holomorphic setting.
 We actually, present three different versions of the Duality Theorem in
 in the context of holomorphic embedding. The precise statements are in:  \ref{prop:firstversion}, \ref{prop:secondver}, \ref{prop:direciso}.
 The original result for this section is  Theorem~\ref{prop:kernforhol} on the structure of holographic operators. Later on, we derive consequences on the structure of symmetry breaking operators. For specific applications of the duality Theorem to branching laws,  see \cite{OV3}.

  A statement for the duality Theorem is:
 \begin{thm}\label{prop:dualfirstv} There exists a linear isomorphism between $Hom_H(V_\sigma,V_\tau)$ and $Hom_L(Z, \mathcal U(\h_0)W)$.
 \end{thm}

 In order to present the elements of the Theorem, to follow, we recall the standard description
on a symmetric complex domain of the holomorphic Discrete Series.   (details in \cite{vDP}, \cite[XII.5]{Neeb2})

 $G$ a semisimple Lie group and a maximal compact subgroup $K$ of $G$.

 $\bigl(G,H=(G^\sigma)_0\bigr)$ is a symmetric pair.

 $H_0:=(G^{\sigma \theta})_0$ associated subgroup to $H$.

 $L=K\cap H$  a maximal compact subgroup of $H$.

 $H/L \rightarrow G/K$   a holomorphic embedding.

 $(L_\cdot^\tau ,V_\tau)$ is a holomorphic Discrete Series for $G$ of lowest $K$-type $(\tau,W)$, realized as holomorphic $W$-valued functions in the bounded symmetric domain $\mathcal D\equiv G/K$.      The action, $L_\cdot^\tau $ of $\g$,  is recalled bellow.

 $(L_\cdot^\sigma ,V_\sigma)$ is a holomorphic Discrete Series for $H$ of lowest $L$-type $(\sigma, Z)$, realized as holomorphic $Z$-valued functions in the bounded symmetric domain $\mathcal D_\h \equiv H/L$.

 Further notation:   $ T$ maximal torus of $K$ so that $U=T\cap H$ is a maximal torus of $L$.  $o :=$the coset $eK$.

 $\Psi$ holomorphic system in $\Phi(\g,\t)$, $\Psi_H, \Psi_{H_0}$ holomorphic systems in $\Phi(\h,\u), \Phi(\h_0,\u)$   so that
 $\sum_{\beta \in \Psi^n} \g_\beta:=\p^+=\p_\g^+= \p_{\h}^+ + \p_{\h_0}^+$ are isomorphic to the respective   holomorphic tangent spaces at $o$. Therefore $G\subset P_+K_\C P_{-}, \mathcal D\subset \p^+, H/L \equiv \mathcal D_\h \subset \p_\h^+, H_0/L \equiv \mathcal D_{\h_0} \subset \p_{\h_0}^+$.

 We recall the triangular  decomposition

 $P_+K_\C P_{-} \ni x=exp (x_+) x_0 exp (x_-), x_\pm \in \p^\pm, x_0 \in K_\C$.\\ The cocycle $c_\tau$ and the reproducing kernel $K_\tau^c$ that defines the Hilbert structure on  $V_\tau$ are:  for $g\in G, z,w \in \mathcal D$,   \\
 \phantom{xxxxxxx} $ c_\tau(g,z)=\tau ((g exp z)_0),\,\, K_\tau^c (w,z)=\tau  (((exp  \bar w)^{-1}exp z)_0)^{-1} $.

 We consider the space $L_\tau^2 (\mathcal D,W):=L_{c_\tau}^2(\mathcal D,W)$ and the subspace
 \begin{equation*} V_\tau :=V_\tau^G= \{ f\in  \mathcal O(\mathcal D,W) : \int_\mathcal D (K_\tau^c (w,w) f(w), f(w))_W dm_{G/K} (w) <\infty \}.  \end{equation*} We recall    the subspace $V_{K-fin}$ of $K$-finite vectors in $V_\tau$ is \begin{equation}\label{eq:1} V_{K-fin}=\mathcal P(\p^+,W)  \equiv \mathcal U(\g)\otimes_{\mathcal U(\k +\p^+)} W \equiv S(\p^-)\otimes W \end{equation}

 Here, we have identified $W$ with the subspace of constant functions.

 The isomorphisms in \ref{eq:1} are $L_D^\tau (w)\leftarrow [D\otimes w] \leftarrow D\otimes w$.

 Here, we use that $\p^\pm$ is an abelian Lie algebra, and hence, the symmetrization map may be thought as the identity map.

 The action $L_g^\tau :=\pi_{c_\tau}(g), g\in G$ in $\mathcal O(\mathcal D,W)$ is defined    by the formula

\phantom{xxxx}$L_g^\tau(f)(z)=\tau ( (g^{-1}exp z)_0)^{-1} f( (g^{-1}exp z)_+), g\in G, z\in \p^+$.

  \subsubsection{}\label{sub:jv} In  \cite{JV}   they have computed the
   the action of $\g$ (resp. $K$) in $\mathcal O(\mathcal D,W)$. For this

 we define $(\delta(x)f)(v)=\lim_{t\rightarrow 0} \frac{ f(v+tx)-f(v)}{t}$, then, for $p\in \mathcal O(\mathcal D,W)$,

 and for $x\in \k, L_x^\tau(p)(v)=\tau (x)(p(v))-(\delta([x,v]) p)(v)$,

 and for $x \in \p^+, L_x^\tau(p)(v)=-(\delta(x) p)(v)$,

 and for $x \in \p^-, L_x^\tau(p)(v)=\tau([x,v])(p(v))-\frac12(\delta([[x,v],v]) p)(v)$,

 and for $k \in K, L_k^\tau (p)(v)=\tau(k)(p(k^{-1}v))$.

 Then, it readily follows that $$V_\tau^{\p^+}=W , \,\,\, V_\tau^{\p_\h^+}=\{p \in V_{K-fin}: \delta(x)p=0, \forall x \in \p_\h^+ \}=\mathcal P(\p_{\h_0}^+,W).$$

 \medskip
 The hypothesis the inclusion $H/L \rightarrow G/K$ is a holomorphic embedding and $V_\tau$ is a holomorphic representation  yields that the restriction of $(L_\cdot^\tau , V_\tau)$ to $H$ is a $H$-admissible representation and the  totality of irreducible factors are holomorphic Discrete Series for $H$ for a reference \cite{KO}. Therefore, after we write $$res_H(V_\tau)=\oplus_{\mu \in Spec(res_H(V_\tau))} V_\tau[V_\mu^H],$$  the subspace $\mathcal L_{W,H}^c:=E_{c_\tau}(\mathcal L_\lambda )$, defined in \cite{Va}\cite{OV3},  equal to the sum of the sum of the lowest $L$-type of the totality of irreducible $H$-factors of $res_H(V_\tau)$ is equal to $\mathcal P(\p_{\h_0}^+,W)$. In fact,

 \begin{equation*} \begin{split} \mathcal L_{W,H}^c =E_{c_\tau}(\mathcal L_{W,H}) \ &  =\oplus_{\mu \in Spec(res_H(V_\tau))} V_\tau[V_\mu^H][V_{\mu^H +\rho_n^H}^L] \\ & =\{p \in V_{K-fin}: \delta(x)p=0, \forall x \in \p_{\h}^+ \}=\mathcal P(\p_{\h_0}^+,W). \end{split} \end{equation*}

 Via the Killing form, $B$, $\p_{\h_0}^-$ is in duality with $\p_{\h_0}^+$, which provides un isomorphism between $\mathcal P(\p_{\h_0}^+,W)$ and $S(\p_{\h_0}^-)\otimes W$. That is, the inverse  map to

 \phantom{xx} $\p_{\h_0}^- \otimes W  \ni Y\otimes w \mapsto (\p_{\h_0}^+ \ni X \stackrel{p_Y}{\longrightarrow} B(X,Y)w)\in \mathcal P(\p_{\h_0}^+,W),$

  extends to a   $L$-equivariant isomorphism

 \begin{equation*} D_0 :    \mathcal P(\p_{\h_0}^+,W) \rightarrow S(\p_{\h_0}^-)\otimes W .\end{equation*} The action of $L$ in $S(\p_{\h_0}^-)\otimes W$ is   tensor product   action.

 We recall $W$ is identify with the constant functions, the equality $\h_0=\p_{\h_0}^- +\l + \p_{\h_0}^+$ and that $V_\tau$ is a holomorphic representation. Next, as in \cite{Va}\cite{OV3} we consider the subspace  \begin{equation*} \begin{split} \mathcal U(\h_0)W  & :=\{L_D^\tau  w, D\in \mathcal U(\h_0), w\in W\} \\ & =\{L_D^\tau  w, D\in \mathcal U(\p_{\h_0}^-), w\in W\}\end{split} \end{equation*} and  the map $D_1$ defined by
 $$\mathcal U(\p_{\h_0}^-)\otimes W \ni D\otimes w \stackrel{D_1}{\longmapsto} L_D^\tau  w \in \mathcal U(\h_0)W.$$

 \subsubsection{}\label{sub:D1iso}Then, \ref{eq:1} yields $D_1$ is a $L$-equivariant isomorphism.

 Finally, we recall $\p_{\h_0}^-$ is an abelian Lie algebra, hence, the symmetrization map from $S(\p_{\h_0}^-)$ onto $\mathcal U(\p_{\h_0}^-)$ may be thought of an identification. Thus, we have given a new proof of \cite{Va} \cite[Proposition 4.9]{OV3} in  the holomorphic setting.

 \begin{prop}\label{prop:D} {\it   The map $D:=D_1 D_0$ is a $L$-equivariant,  degree preserving, linear isomorphism between $\mathcal L_{W,H}^c$ and $\mathcal U(\h_0)W$.}
 \end{prop}
 Obviously, $D_0$ preserves the degree of polynomials.  $D_1$ preserves the degree of polynomials owing to the explicit formula for $L_x^\tau , x\in \p^-$,   hence  $L_R(w)$ is a polynomial of the same degree as the degree of $R\in \mathcal U(\p^-)$.

 The map $D$ in Proposition~\ref{prop:D} is an example of  the map $D^c$ in \ref{sub:dualinsymm}.
  \subsubsection{Statement and proof of the first version of duality} \label{sub:first} We are ready to provide a  statement and  proof of the first version of the duality theorem for holomorphic setting.

 For this, we assume the inclusion $H/L \rightarrow G/K$ is holomorphic. Let $(L_\cdot^\tau , V_\tau)$ be a holomorphic Discrete Series representation for $G$ realized in a Hilbert subspace of  $\mathcal O(\mathcal D,W)$. Let $V_\sigma$ a holomorphic Discrete Series for $H$ of lowest $L$-type $(\sigma, Z)$,  realized as a Hilbert space of holomorphic $Z$-valued functions on  the bounded domain $  \mathcal D_\h \subset \p_\h^+$ associated to $H/L$. Thus, $res_H(V_\tau)$ is a $H$-admissible representation.    Each continuous linear,  $H$-map, $T$, from  $V_\sigma $ into $V_\tau$ is represented by a kernel $K_T :\mathcal D_\h \times \mathcal D \rightarrow Hom_\C (Z,W)$. Among the properties of $K_T$ we recall $\text{for}\, x\in \mathcal D, w\in W, K_T(\cdot, x)^*w \in V_\sigma$; $\text{for}\, z \in Z, K_T(o, \cdot)z \in V_\tau[V_\sigma][Z] \subset \mathcal L_{W,H} \subset V_{K-fin}$; the map $Z\ni z \rightarrow K_T(o,\cdot)z \in V_\tau$ is a $L$-map.
 \begin{thm}\label{prop:firstversion} Under the previous assumptions, we have a linear isomorphism between $Hom_H(V_\sigma,V_\tau)$ and $Hom_L(Z, \mathcal U(\h_0)W)$. An isomorphism is the map $$ Hom_H(V_\sigma,V_\tau)\ni T \mapsto \big(Z\ni z \mapsto D(K_T(o,\cdot)z)(\cdot)\in Hom_L(Z, \mathcal U(\h_0)W)\big).$$
 \end{thm}
 \begin{proof} We recall a Discrete Series representation for $H$ is determined by its lowest $L$-type. The following computation shows the two involved spaces have the same dimension.

 \noindent
 $\dim Hom_H(V_\sigma, V_\tau)=\dim \mathcal L_{W,H}^c [Z] $

 \phantom{xxxxxxxxxxxxxxxxxx}$\stackrel{D}{=}\dim \mathcal U(\h_0)W[Z]=\dim Hom_L(Z, \mathcal U(\h_0)W)$.

 The map $T \mapsto D(K_T(e,\cdot))$
  is injective, owing to  $D$ is injective  and the equality
  $K_T(h\cdot o, w)=c_\sigma (..,..)K_T(o, h^{-1}\cdot w) c_\tau  (..,..)$.
 Equality of dimensions shows the surjectivity.
 \end{proof}

 We suggest to confront   Theorem~\ref{prop:firstversion} with \cite[Theorem 3.10]{Na}\cite[Theorem 3.1]{OV3} and subsection~\ref{sub:DualityinG}.

 \subsubsection{Multiplicity formulae} The duality Theorem provides a formula for $\dim Hom_H(H^2(H,\sigma), H^2(G,\tau))$ in terms of a partition function based on the noncompact roots for $(\h_0, \u)$ and a Weyl group, see  \cite{DV}, \cite{OV3} and references therein. In \cite{GW} we also find a multiplicity formula based on a partition function and Weyl groups. In \cite[Theorem 8.3]{Kob3} the author computes the Harish-Chandra parameters of all  the factors in $res_H(V_\tau)$. In \cite{HHO} it is shown a multiplicity formula for either weak factors or factors based on the map $r_n$ considered by \cite{OV}.  On the other hand,  Paradan, \cite{Pa}, Duflo and  Vergne, \cite{DVer}   have obtained $\dim Hom_H(V_\sigma, V_\tau)$ as the  volume of a orbifold.

 \subsection{Global structure of $\mathcal U(\h_0)W$ and second version of the Duality Theorem} We decompose $W=Z_1 +\cdots + Z_r$,  with $Z_j$ a $L$-invariant, $L$-irreducible linear subspace. Then,
 since for $ x\in \p_{\h_0}^+, w \in W$ we have $L_x^\tau w=0$. It readily follows that $\mathcal U(\h_0)Z_j$ is the underlying Harish-Chandra module of the irreducible square integrable representation of lowest weight $Z_j$. Therefore,
 \subsubsection{}\label{sub:5} $\mathcal U(\h_0)W =\oplus_{1\leq j \leq r} \,\mathcal U(\h_0)Z_j \equiv \oplus_{1\leq j\leq r} \, \mathcal U(\h_0)\otimes_{\mathcal U(\l_\C+ \p_{\h_0}^+)}  Z_j$.

 Whence, we have obtained the decomposition of $\mathcal U(\h_0)W$ as the sum of underlying Harish-Chandra modules of Discrete Series representations for $H_0$. The conclusion of Theorem~\ref{prop:firstversion} can be written as,

 \smallskip

\begin{thm}
 \label{prop:secondver}
$$ Hom_H(V_\sigma, V_\tau)\equiv \oplus_{1\leq j\leq r} \, Hom_L(Z, \mathcal U(\h_0)\otimes_{\mathcal U(\l_\C+ \p_{\h_0}^+)} Z_j).$$
\end{thm}

 \smallskip
  We would like to point out that Theorem~\ref{prop:secondver} follows from  facts proven  in Kobayashi-Pevzner \cite[Theorem 2.7]{KP2}, Nakahama \cite{Na}[Lemma 3.4, Theorem 3.6].  Our proof is of algebraic nature. \\
 Owing to our hypothesis $\mathcal D_{\h_0}$ is a subset of $\mathcal D$ and the inclusion is holomorphic,  we may restrict   holomorphic function on $\mathcal D$ and we obtain holomorphic functions on  $\mathcal D_{\h_0}$. Let $r_0 : \mathcal O(\mathcal D,W)\rightarrow \mathcal O(\mathcal D_{\h_0},W)$ denote the restriction map. In \cite{OV} it is shown that $r_0$ maps $L^2$-continuously the space $V_\tau$ into $L_{c_\tau }^2(\mathcal D_{\h_0},W)$.  In \cite[Lemma 3.6]{OV3}  it is
 shown that the kernel of $r_0 $ is equal to the orthogonal to the subspace $\mathrm{Cl}(\mathcal U(\h_0)W)$.  Whence, $r_0 : \mathrm{Cl}(\mathcal U(\h_0)W)\rightarrow  L_{c_\tau }^2(\mathcal D_{\h_0},W)$ is injective.
 Now, $L_{c_\tau }^2(\mathcal D_{\h_0},Z_j))$ contains just once the holomorphic Discrete Series
 of lowest weight $Z_j$. Let's denote by $H_{c_\tau }^2(\mathcal D_{\h_0},Z_j)$ such a subspace. Actually, $H_{c_\tau }^2(\mathcal D_{\h_0},Z_j)$ is the kernel of the $\bar{\partial}$ operator in $L_{c_\tau}^2(\mathcal D_{\h_0},Z_j))$
 As in \cite{OV3}, we define
 $$\mathbf H_{c_\tau}^2(\mathcal D_{\h_0}, W):=\oplus_{1\leq j \leq r} H_{c_\tau}^2(\mathcal D_{\h_0},Z_j)\subset L_{c_\tau}^2(\mathcal D_{\h_0}, W)$$
 Then, from the previous calculations, it readily follows that the image
 of $r_0$ is $\mathbf H_{c_\tau}^2(\mathcal D_{\h_0}, W)$. Finally, we define \begin{equation} \label{eq:r0D} r_0^D : Hom_H(V_\sigma, V_\tau) \rightarrow  Hom_L(Z,\mathbf H_{c_\tau}^2(\mathcal D_{\h_0}, W))  \end{equation}

 by the rule \begin{equation*} r_0^D(T)(z)= r_0(D(K_T(o,\cdot)z)). \end{equation*}
 Obviously,  $r_0^D$ is a linear bijection. The isomorphism $r_0^D$ is our second version of the duality Theorem.

 \subsection{Another equivalence map for $\mathcal L_{W,H}^c\equiv \mathcal U(\h_0)W $ and third version of Duality Theorem} To begin with, we show, we may replace the map $D$ defined in Proposition~\ref{prop:D} by the orthogonal projector $Q$  onto $\mathrm{Cl}(\mathcal U(\h_0)W)$. That is, \begin{prop}\label{prop:qisd} Let $Q :V_\tau \rightarrow V_\tau $ denote the orthogonal projector onto $\mathrm{Cl}(\mathcal U(\h_0)W)$. Then, \\
 \phantom{xxxxxxxxxxxxxx}{\it $Q : \mathcal L_{W,H}^c \rightarrow \mathcal U(\h_0)W$ is a $L$-equivalence.}
 \end{prop}
 \begin{proof} Indeed, to begin with we recall     $\mathcal P(\p^+,W)\subset V_\tau=\mathcal O(\mathcal D,W)\cap L_{\tau}^2(\mathcal D,W)$. Here, $\mathcal L_{W,H}^c=\mathcal P(\p_{\h_0}^+,W)$. Thus, $r_0(p)=p, \, \forall p \, \in \mathcal P(\p_{\h_0}^+,W)$.  We write $V_\tau\ni p=p_1 +Q(p)$ with $p_1 \in \mathrm{Cl}(\mathcal U(\h_0)W)^\perp, Q(p) \in \mathrm{Cl}(\mathcal U(\h_0)W)$. Since, $Ker(r_0)=\mathrm{Cl}(\mathcal U(\h_0)W)^\perp$ \cite{OV3}, we have
 \begin{equation*} r_0(p)=r_0(Q(p)) \, \forall p \in V_\tau, \end{equation*} hence, $Q$ restricted to $\mathcal L_{W,H}^c=\mathcal P(\p_{\h_0}^+,W)$ is injective. Since $Q$ is a $L$-map, the $L$-admissibility of $V_\tau$, \ref{prop:D} and that $\mathrm{Cl}(\mathcal U(\h_0)W)_{L-fin}= \mathcal U(\h_0)W$ we obtain $Q$ restricted to $\mathcal L_{W,H}^c$  is a bijection onto $ \mathcal U(\h_0)W$. Since the equivalences between the different models are unitary intertwining operators, it readily follows the general statement. \end{proof}

 We believe the statement holds for arbitrary symmetric pairs $(G,H)$ under the hypothesis of $H$-admissibility.

 A consequence of the previous Proposition is an improvement  of the duality Theorem.

 \begin{thm}
 \label{prop:direciso}\phantom{xxxxxxxxxxxxxxx}

  a) For $T_1, T_2 \in Hom_H(V_\sigma, V_\tau)$, we have   $K_{T_1}= K_{T_2}$ if and only if \\ \phantom{xxxxxxxxxxxxxxxxxxx} $K_{T_1}(e,h_0)z=K_{T_2}(e,h_0)z   \,\,\,\forall z \in Z, h_0\in H_0$.

 b) The map \\  $Hom_H(V_\sigma, V_\tau) \ni T $

  $ \mapsto (Z\ni z \mapsto  r_0(K_T(e,\cdot)z)(\cdot) \in \mathbf H^2(H_0,\tau)))\in Hom_L(Z, \mathbf H^2(H_0,\tau)) $ \\ is a bijection.
 \end{thm}

Actually, $a)$ can be extracted from  Nakahama in \cite[Lemma 3.4]{Na}. Statement $b)$ can be deduced from    \cite[Lemma 3.4]{Na} or the work of  Kobayashi \cite[Proof of Lemma 8.8]{Kob3}. Their proofs apply quite different techniques.

 In \cite{OV3} we have computed the orthogonal projector $Q$ in terms of the operator $r_0^\star r_0$, for example, when $\tau$ restricted to $L$ is irreducible, $Q$ is up to a constant the operator $r_0^\star r_0$, whence, applying the results in \cite{OV3} we may compute the inverse to the map defined in part b) of the Theorem. In general, under the hypothesis of $H$-admissibility, in \cite[Theorem 3.1]{OV3} it is shown that $Im(r_0)$ is a closed subspace of $L^2(H\times_\tau W)$, actually, it is a finite sum or irreducible subrepresentations, and that  $r_0r_0^\star $ is a  bijective linear operator  for $Im(r_0)$. It can be checked that $Q=r_0^\star (r_0 r_0^\star)^{-1} r_0$.

 \smallskip
 Due to the equality $r_0 =r_0 Q$,   we obtain a simpler description of inverse map in Theorem~\ref{prop:direciso} b). In fact,
 formula \ref{sub:DualityinG} simplifies to \begin{equation*}  K_T(h,x)z      =  \big(Q^{-1}\big[\int_{H_0} K_\tau (h_0,\cdot)      (C(r_0 (K_T(e,\cdot)z))(\cdot))(h_0)dh_0 \big]\big)(h^{-1}x).
  \end{equation*} Here, $C:=(r_0 r_0^\star)^{-1} $ is a bijective endomorphism  of $\mathbf{H}^2(H_0,\tau)$.
 \bigskip

 \subsection{Some versions of the inversion formula} In Theorem~\ref{prop:direciso} we have presented a new proof that the map $$ Hom_H(V_\sigma, V_\tau) \ni T \mapsto (z\rightarrow r_0(K_T(e,\cdot)z)(\cdot))\in Hom_L(Z,\mathbf H^2(H_0\times_\tau W)).$$ is bijective. Actually, we have   shown injectivity and that both spaces are equidimensional. In this manner, we have obtained the surjectivity of the map.  Nakahama \cite{Na}(see Proposition~\ref{prop:nakahama}) have written an explicit   inverse  map based on the reproducing kernel of the target space $V_\tau$ and the realization of holomorphic Discrete Series in spaces of functions on the bounded symmetric domain attached to the group $G$.  In the next paragraph we  present Nakahama's formula, and,  in \ref{prop:kernforhol}, we show another  expression for the inverse map, based on the reproducing kernel of the initial factor $V_\sigma$ that we are considering  and a realization of the Discrete Series representation on space of functions on $G$.

 \subsubsection{Nakahama's formula}   The hypothesis is: both representations are realized in a space of holomorphic functions. Thus,   $V_\tau=\mathcal O(\mathcal D,W)\cap L_\tau^2(\mathcal D, W)$ is a $H$-admissible holomorphic
 Discrete Series representation for $G$. $V_\sigma$ is a holomorphic Discrete Series for $H$ realized on a space of holomorphic functions on $\mathcal D_\h$. We also assume the inclusion $\mathcal D_\h \equiv H/L \rightarrow G/K \equiv\mathcal D $ is holomorphic map. Whence, we may consider  the triangular decomposition,\\
   $P_+K_\C P_- \ni  x=exp(x_+) \,(x)_0 \, exp(x_-), x_\pm \in Lie(P_\pm) =\p^\pm$, $(x)_0 \in K_\C$.\\ For $g \in G, x\in \mathcal D,\, g\cdot x := (g\, exp(x))_+$. Since $H$ is the fixed point set of an involution that commutes with the Cartan involution determinate by $K$, we have that the triangular decomposition $P_+K_\C P_- $ is invariant under the involution that defines $H$.  Therefore, $P_+K_\C P_- \cap H=: P_+^H L_\C P_-^H$ and the triangular decomposition of $h\in H$ is equal to the triangular decomposition of $h$ thought in $G$. This also holds for the action of $H$ in $\mathcal D_\h \subset \mathcal D$.  We consider $T :V_\sigma \rightarrow V_\tau$ a continuous $H$-map, and its kernel $K_T : \mathcal D_\h \times \mathcal D \rightarrow Hom_\C(Z,W)$. Hence, it holds the equality \begin{equation}\label{eq:forK_T}  K_T^c(h\cdot w,h \cdot  z)=c_\tau (h,z)K_T^c (w,z) c_\sigma (h,w)^\star, \, \forall h \in H, w \in \mathcal D_\h, z\in \mathcal D, \end{equation}

 Here, owing to our hypothesis we have,

 $c_\tau (g,z)=\tau ( (g exp z)_0), c_\sigma (h,w)=\sigma( (h exp w)_0), l\cdot z=Ad(l)z.$

 Hence, for $l \in L, z\in \mathcal D \subset \p^+ $, we obtain
 \begin{equation} \label{eq:lk_t}   K_T^c(o, Ad(l)(z))=K_T^c( l \cdot o,  l \cdot z)=   \tau (l) K_T^c(o,z)\sigma (l^{-1}).   \end{equation} That is $K_T(o, \cdot) \in Hom_L(Z, \mathcal O(\mathcal D)\otimes W)$.

  Writing $w=h_w\cdot o, h_w\in H$, we obtain the equality
 \begin{equation}  K_T^c(w ,z)=c_\tau (h_w, (h_w)^{-1}\cdot z)K_T^c (o,(h_w)^{-1}\cdot z) c_\sigma (h_w,o)^\star. \end{equation}

 Therefore, the kernel $K_T^c$ is determined by the functions
 \begin{equation} \{ z\mapsto K_T^c(o,\cdot)z, z\in Z\}. \end{equation}

 Actually, this can improved.  For this we recall the orthogonal decomposition   $\p_\g^+ =\p_{\h_o}^+ \oplus \p_{\h}^+ $ and  denote by $P_2$  the projector onto $\p_{\h_0}^+$ along $\p_{\h}^+$.

  \subsubsection{}\label{sec:r0injec} Let, as usual, be  $r_0^c : V_\tau \rightarrow L_{c_\sigma}^2 (\mathcal D_{\h_0},W)$ the restriction map, it is obvious that $r_0^c : \mathcal L_{W,H}^c =\mathcal P(\p_{h_0}^+,W) \rightarrow  \mathcal P(\p_{h_0}^+,W)\subset L_{\sigma}^2 (\mathcal D_{\h_0},W)$ is the "identity" map. In other words, $r_0^c$ restricted to $ \mathcal L_{W,H}^c $ is injective. Whence, we obtain a new proof of Theorem~\ref{prop:direciso} a), namely,  the kernel $K_T^c$ is determined by the function

   \phantom{xxxxxx} $Z\ni z \mapsto r_0^c(K_T^c(o,\cdot)z)  \in  \mathcal P(\p_{h_0}^+,W)\subset  L_{\sigma}^2 (\mathcal D_{\h_0},W) $.

  Equivalently, the following equality hods $$ K_T^c(o,\cdot )=K_T^c(o,P_2(\cdot ) ).$$

 Let $x\mapsto \bar x$ denote the conjugation associated to the pair $(\g_\C, \g)$. Then, the conjugation carries $\p^+$ onto $\p^-$.

 In \cite[Proposition 3.3]{Na} we find a proof of,
 \begin{prop} [Nakahama's formula]\label{prop:nakahama} Let $T: V_\sigma \rightarrow V_\tau$ a continuous $H$-map. Then,  for $  w\in \mathcal D_\h, z\in \mathcal D$,  \begin{equation}\label{eq:naka}  K_T^c(w ,z)=K_\tau^c (w, z)\,\,  K_T^c (o, P_2 \bigl(\big(exp(-\bar w)exp(z)\big)_+\bigl)). \end{equation}
 \end{prop}
  Thus, Nakahama  provides an explicit inversion formula to the map in Theorem~\ref{prop:direciso} b).

 {\it Note.} Nakahama's  statement as well as his proof is    written  in   the language of Jordan algebras. We are able to obtain  a quite long proof of the Proposition without the terminology of Jordan algebras, however, it is just a copy, without Jordan algebras language, of the techniques of Nakahama. Our proof replaces,   the result Nakahama needs from \cite[Part V, Thorem III.5.1(ii)]{FK}, by    facts  in \cite[Lemma XII.1.8]{Neeb2} \cite[Chapter II, $\S$ 5]{Sa}.

  \subsubsection{} Nakahama´s formula for the kernel of a holographic operator together with Theorem~\ref{prop:propertiesksc} let us compute the kernel of a  symmetry breaking operator. The final result is:
 \begin{prop} We assume $H/L \rightarrow G/K$ is a holomorphic embedding, and $V_\tau=\mathcal O(\mathcal D,W)\cap L_\tau^2(\mathcal D, W)$ (resp. $V_\sigma=\mathcal O(\mathcal D_\h,Z)\cap L_\tau^2(\mathcal D_\h, Z)$) is a holomorphic Discrete Series representation of $G$ (resp.is a holomorphic representation for $H$). Then, for each $S \in Hom_H(V_\tau, V_\sigma)$, and the function $\Phi_S=K_{S^\star}^c(o,\cdot) \in  Hom_L(Z, \mathcal P(\p_{\h_0}^+,W))= (\mathcal P(\p_{\h_0}^+)\otimes Hom_\C (Z,W))^{(L_\cdot \otimes \sigma^\star \otimes \tau) (L)}  $  we have \begin{multline*} K_S^c(z,w)(\cdot)\\ = \big(\Phi_S(P_2 \bigl((exp(-\bar w)exp(z)\big)_+\bigl)\big)^\star (K_\tau^c (z,w)(\cdot) )   , z\in \mathcal D, w \in \mathcal D_\h. \end{multline*} Actually, for each    $\Phi  \in  Hom_L(Z, \mathcal P(\p_{\h_0}^+,W)) $, there is a unique symmetry breaking operator $S$ so that $\Phi_S=\Phi$.
 \end{prop}
 When we write $\Phi_S(v)(z_0)=\sum_r C_r(v) z_0^r, v\in Z, C_r\in Hom_\C(Z,W)$, $z_0\in \p_{\h_0}^+$, then, $\big(\Phi_S(z)\big)^\star=\sum_r C_r^\star \,\bar{z_0}^r$.
 \subsubsection{Nakahama's formula in $H^2$-model} The hypothesis $H/L\rightarrow G/K$ is a holomorphic embedding is in force, $H^2(H,\sigma), H^2(G,\tau)$ are respective holomorphic Discrete Series representations. \begin{prop} Let $T :H^2(H,\sigma)\rightarrow  H^2(G,\tau)$ be a holographic operator. Then, \begin{equation*} K_T(h,x)=     K_\tau (h, x) \,  \, K_T(e,  (h)_0^{\star }\, (g_a \,(g_a)_0^{-1})\,  (h)_0^{\star-1}     ).  \end{equation*} \end{prop} \begin{proof} To begin with we recall that for $z\in \mathcal D$,  the unique $g_z \in exp(\p)\cap G$ so that $g_z\cdot o=z$ is $g_z=exp(z)K_\tau^c(z,z)^{1/2} exp(\bar z)$. We note $K_\tau^c(z,z)^{1/2}$ is well defined due to the identity $K_\tau^c(w,z)=\overline{K_\tau^c (z,w)}^{-1}$, and in consequence,   $K_\tau^c(z,z)$ belongs to the $exp(\p)$-part of $K_\C=K  exp(iLie(K))$ so $log(K_\tau^c(z,z))$ is defined.
 \cite[Chap II, $\S$ 5, Exercise 2]{Sa}\cite{Neeb2}.

 \bigskip
  For $w\in \mathcal D_\h, z \in \mathcal D$, we write (at least for small $w,z$ it is true)\\ $P_2((exp(-\bar w)exp(z))_+)=g_a \cdot o$, $a\in \mathcal D_{\h_0}, g_a \in exp(\p \cap \h_0) $. Now, we apply \ref{eq:ttc} and Nakahama's formula \ref{eq:naka} to obtain
 \begin{multline*}K_T(g_w,g_z)=K_\tau (g_w,g_z) \, \tau ((g_w)_0)^\star \tau ((g_a)_0)\, K_T(e, g_a)\, \sigma ((g_w)_0)^{\star-1}\\
  = K_\tau (g_w,g_z) \,  \, K_T(e, ((g_w)_0)^{\star }\, g_a \,  \big(((g_w)_0)^\star    (g_a)_0\big)^{-1})\\  = K_\tau (g_w,g_z) \,  \, K_T(e, ((g_w)_0)^{\star }\, (g_a \,(g_a)_0^{-1})\,  ((g_w)_0)^{\star-1}     ). \end{multline*}
  The second equality is due to $K_T(e,lxk)=\tau(k^{-1})K_T(e,x)\sigma(l^{-1})$.
 Owing to  real analyticity of both sides in the formula, the equality holds for every $w,z$.  Next, we write $h=g_{h\cdot o} \, l, x=g_{x\cdot o}\, k$, and compute
 \begin{multline*}K_T(h,x)=\tau(k^{-1}) K_T(g_{h\cdot o},g_{x\cdot o})\sigma(l) \\ = \tau(k^{-1}) K_\tau (g_{h\cdot o},g_{x\cdot o})) \, \tau ((g_{h\cdot o})_0)^\star \tau ((g_a)_0)\, K_T(e, g_a)\, \sigma ((g_{h\cdot o}))_0)^{\star-1} \sigma(l)\\
  = K_\tau (g_{h\cdot o}l ,g_{x\cdot o}k) \, \sigma(l^{-1}) \, K_T(e, \sigma(l^{-1})((g_{h\cdot o} )_0)^{\star }\, g_a \,  \big(((g_{h\cdot o})_0)^\star    (g_a)_0\big)^{-1})\\ \,  = K_\tau (h, x) \,  \, K_T(e, \sigma(l^{-1})(g_{h\cdot o})_0)^{\star }\, (g_a \,(g_a)_0^{-1})\,  ((g_{h\cdot o}))_0)^{\star-1}\sigma(l)    ) \\ \,  = K_\tau (h, x) \,  \, K_T(e,  (h)_0^{\star }\, (g_a \,(g_a)_0^{-1})\,  (h)_0^{\star-1}     ). \end{multline*}
 \end{proof}
 \noindent
 We note that $((g_w)_0)^{\star }\, (g_a \,(g_a)_0^{-1})\,  ((g_w)_0)^{\star-1}=e^r e^s,\, r\in \p_{\h_0}^+, s\in \p_{\h_0}^-$.

 When both representations are scalar the formula  turns into
 $$K_T(g_w,g_z)=  \tau ((g_w)_0)^\star \tau ((g_a)_0)\,\, \sigma ((g_w)_0)^{\star-1}\, K_\tau (g_w,g_z) \, K_T(e, g_a) $$

 \subsection{A formula for     holographic operator's in the $H_{hol}^2$-model} In Theorem~\ref{prop:direciso} we have presented a new proof that the map $$ Hom_H(V_\sigma, V_\tau) \ni T \mapsto (z\rightarrow r_0(K_T(e,\cdot)z)(\cdot))\in Hom_L(Z,\mathbf H^2(H_0\times_\tau W)).$$ is bijective. Actually, we have   shown injectivity and that both spaces are equidimensional.  Nakahama (see Proposition~\ref{prop:nakahama}) have written an explicit   inverse  map based on the reproducing kernel of the target space $V_\tau$ and the realization of holomorphic Discrete Series in spaces of functions on the bounded symmetric domain attached to the group $G$.  In the next paragraph we will present another  expression for the inverse map, based on the reproducing kernel of the initial factor $V_\sigma$ that we are considering  and a realization of the Discrete Series representation on space of functions on $G$.    In order to present the formula we   recall a fact shown by Bailey-Borel and another realization for the holomorphic Discrete Series.  We assume $G/K$ is isomorphic to a bounded symmetric domain $\mathcal D \subset \p^+$. Let $(\tau, W)$ denote the lowest $K$-type of a holomorphic Discrete Series representation realized in $V_\tau =\mathcal O(\mathcal D,W)\cap L_\tau^2(\mathcal D,W)$. Then, we have triangular decomposition $G\subset P^+K_\C P^-$, $x=exp(x_+)(x)_0 exp(x_-)$,  the cocycle  $c_\tau (x,w)=\tau ((xexp(w))_0)$, $x\in G, w\in \mathcal D$, the reproducing kernel for $V_\tau$ is $K_\tau(w,z)=\tau ((exp(-\bar w)exp(z))_0^{-1})$. We also have the  linear isomorphism $E_{c_\tau} : C^\infty (G\times_\tau W)\rightarrow C^\infty (\mathcal D,W)$ defined by $E_{c_\tau} (f)(xK)=c_\tau(x,o)f(x)$ (see \ref{eq:difop}). Then, Bailey-Borel \cite{BB} have shown the pre-image of the space of holomorphic functions  $\mathcal O(\mathcal D, W)$,  via the map $E_{c_\tau}$, is equal to the subspace     \begin{multline} \mathcal O(G \times_\tau W):=\{ f :G\rightarrow W, smooth, \\ f(xk)=\tau(k^{-1}) f(x),\,k\in K, x\in G,   R_X(f)=0 \,\, \forall X \in \p^- \}.\end{multline}
 We also consider similar isomorphisms for the domains $H/L, H_0/L$.
 Now,  we may    realize holomorphic Discrete Series representations in spaces of "holomorphic" sections in $\Gamma^\infty (G\times_\tau W)$. That is,
 given $(\tau,W)$ the lowest $K$-type of a {\it Holomorphic} Discrete Series representation, we consider the realization \begin{multline} H_{hol}^2(G \times_\tau W)=\{ f: G\rightarrow W : C^\infty,  f(xk)=\tau(k^{-1}) f(x),\,k\in K,    x\in G, \\ \int_G \vert f(x)\vert^2 dx <\infty,\,\text{and}\,  R_X(f)=0, \forall X \in \p^-\}. \end{multline}
 In a similar way, we realize the  holomorphic Discrete Series representations for $H$ or $H_0$. In order to justify this realization of Holomorphic Discrete Series we recall that the equivalence $E_{c_\tau}$ is a isometry from $L^2(G\times_\tau W)$ onto $L_\tau^c(\mathcal D, W)$,    see section~\ref{sec:sym(holo)} and references therein.

 We recall that  $(G,H)$ is a  symmetric pair,        $H_0=(G^{\sigma \theta})_0$, and the generalized Cartan decomposition, this is the smooth decomposition $G=exp(\h\cap \p)exp(\h_0\cap \p)K$.  For $  x \in G$, we write the corresponding unique decomposition $x=x_1x_2k$.

 A aim of this subsection is to show \begin{thm}  The kernel $K_T$ of each holographic operator $ T$ for  holomorphic Discrete Series presented in $H_{hol}^2$-model, decomposes as the composition "separation of variables formula" \begin{multline*}K_T(h, x_1x_2k)(z) =\tau(k^{-1})K_T(e,x_2)(K_\sigma (h,x_1)z),\\ h\in H, x_1 \in exp(\h \cap \p), x_2 \in exp(\h_0 \cap \p), k\in K. \end{multline*}.
 \end{thm}

 For this we consider    a holomorphic Discrete Series $H_{hol}^2(H \times_\sigma Z)$ for $H$ and we fix $\Phi : Z\rightarrow \mathcal O(H_0 \times_\tau W)$. We define $T_\Phi : \mathcal O(H\times_\sigma Z) \rightarrow C^\infty (G \times_\tau W)$ by the rule: For $x=x_1x_2k,  x_1\in exp(\h\cap \p), x_2\in exp(\h_0\cap \p),  k \in K$
  \begin{equation} \label{eq:tmultipli} T_\Phi(g)(x):=\tau(k^{-1})\Phi (g(x_1))(x_2). \end{equation}

 We claim: {\it $T_\Phi$ is $H$-intertwining map for the   respective  left  translation actions if and only if}   \begin{equation} \label{eq:phi}  \Phi (\cdot)(lyl^{-1}) =\tau(l)\Phi(\sigma(l^{-1})(\cdot))(y),   \forall \, l\, \in L, y \in exp(\h_0\cap \p).\end{equation}

 \smallskip
 In fact,  we fix $h\in H$. We write $hx_1=x_3 l_3$ with $x_3 \in exp(\h\cap \p), l_3 \in L$,   $hx_1x_2k=x_3 l_3x_2l_3^{-1}l_3k$.  Then, the equality  $T_\Phi(g)(hx)=T_\Phi(L_{h^{-1}}g)(x)$ is equivalent to the  equalities
 \begin{equation*}
 \begin{split}
  \tau(k^{-1})\tau(l_3^{-1})\Phi( g(x_3))(l_3x_2l_3^{-1}) & =\tau(k^{-1})\tau(l_3^{-1})\Phi(g(hx_1l_3^{-1}))(l_3x_2l_3^{-1}) \\    & =\tau(k^{-1})\tau(l_3^{-1})\Phi(\sigma (l_3)g(hx_1))(l_3x_2l_3^{-1}) \\  & \stackrel{?}=\tau(k^{-1})\Phi(g( hx_1))(x_2). \end{split}\end{equation*}
 Since,   the set  $linspan_\C \{ g(hx_1): g \in H_{hol}^2(H,\sigma)[Z]    \}$ is equal to $Z$,  we get the equivalence.

 \smallskip We would like to point out that   condition \ref{eq:phi} on $\Phi$ is equivalent to   $\Phi$ belongs to $  Hom_L( Z,  C^\infty (H_0 \times_\tau W))$.

 \subsubsection{Injectivity of $r_0$ restricted to $\mathcal L_{W,H}= \mathcal L_\lambda$} \label{sec:r0injec2}In \ref{sec:r0injec} we have shown the restriction map $r_0^c : C^\infty (\mathcal D,W)\rightarrow C^\infty (\mathcal D_{\h_0},W)$ when restricted to $\mathcal L_{W,H}^c$ is injective. The commutativity of the diagram below, shows the restriction map $r_0 : \Gamma^\infty (G\times_\tau W)\rightarrow \Gamma^\infty (H_0\times_\tau W)$ restricted to $\mathcal L_{W,H}=\mathcal L_\lambda$ is one to one.

 \xymatrix{  {\mathcal L_{W,H}}\ar[r]^{=} \ar@<1ex>[d]^{\subset}  &  {\mathcal L_{W,H}}\ar[r]^-{\subset} \ar@<1ex>[d]^{E_{c_\tau}} &{H_{hol}^2(G\times_\tau W)\subset L^2(G\times_\tau W)} \ar@<1ex>[d]^{E_{c_\tau}}            \\
   {H_{hol}^2(G\times_\tau W)}\ar@<1ex>[d]^{r_0} &  {\mathcal L_{W,H}^c}\ar[r]^{\subset}\ar@<1ex>[d]^{r_0^c}     & {V_\tau \subset L_\tau^2 (\mathcal D,W)}  \ar@<1ex>[d]^{r_0^c}        \\
    {L^2(H_0 \times_\tau W)}\ar[r]^{E_{c_\tau}}& {L_\tau^2(\mathcal D_{ h_0},W)}\ar[r]^{=}
  & {L_\tau^2(\mathcal D_{ h_0},W)}  .}

  Here, $\mathcal L_{W,H}:=\mathcal L_\lambda $ is the linear span of the totality of  lowest $L$-type subspaces on each irreducible factor for $res_H(H_{hol}^2(G\times_\tau W))$, that is,\\ $\mathcal L_{W,H}=\oplus_{H_{hol}^2(H\times_\sigma  Z) \in Spec_H(H_{hol}^2(G\times_\tau W)) } H_{hol}^2(G\times_\tau W) [H_{hol}^2(H\times_{\sigma } Z)][Z]$.  $V_\tau =\mathcal O(\mathcal D, W)\cap L_\tau^2 (\mathcal D,W)$.  $\mathcal L_{W,H} =\mathcal P(\p_{\h_0}^+,W)$ is the linear span of the lowest $L$-type subspaces on each $H$-isotypic component for $res_H(V_\tau)$. $ E_{c_\tau}$ is the map $f\mapsto c_\tau(\cdot,e)f(\cdot)$. $ E_{c_\tau}$ establishes a $H$-bijection between $H_{hol}^2(G\times_\tau W)$ (resp. $L^2(G\times_\tau W)$)  and $V_\tau$ (resp. $L_\tau^2 (\mathcal D,W)$). $\mathcal L_{W,H}^c=E_{c_\tau}(\mathcal L_{W,H})$.

   \subsubsection{}  To follow, we present a new expression for holographic operators.
   The hypothesis are: $(G,H)$ is a  symmetric pair;       We also assume $H/L \rightarrow G/K$ is a holomorphic embedding, and $H_{hol}^2(G\times_\tau W)$ (resp. $H_{hol}^2(H\times_\sigma Z)$) is a holomorphic Discrete Series representation of $G$ (resp. is a holomorphic representation for $H$).  We recall $H_0=(G^{\sigma \theta})_0$ and the "refined Cartan decomposition", that is,
   the smooth decomposition $G=exp(\h\cap \p)exp(\h_0\cap \p)K$.   \begin{thm}\label{prop:kernforhol} For each holographic operator  $T \in Hom_H(H_{hol}^2(H\times_\sigma Z), H_{hol}^2(G\times_\tau W)$ there exists $\Phi \in Hom_L(Z, H_{hol}^2(H_0 \times_\tau W))$ so that for every $g\in H_{hol}^2(H\times_\sigma Z)$, $G\ni x=x_1x_2k, x_1 \in exp(\h \cap \p),  x_2 \in exp(\h_0 \cap \p), k\in K$,  we have $$T(g)(x)=T_\Phi(g)(x)=\tau(k^{-1})\Phi (g(x_1))(x_2).$$ Whence,  its kernel $K_T$ is equal to the function \begin{equation*}K_T(h,x)z= \tau(k^{-1})\Phi ( K_\sigma (x_1,h)^\star z)(x_2) = \tau(k^{-1})\Phi  ( K_\sigma (h,x_1)z)(x_2).\end{equation*}
   \end{thm}
   Here, $ K_\sigma (h,x_1)z$ is the reproducing kernel for $H_{hol}^2(H \times_\sigma  Z)$.

\smallskip

\noindent
Theorem~\ref{prop:kernforhol} generalizes \cite[Theorem 5.1]{Na}.

  \begin{proof} For $H_0 \ni h_0= x_2 l, x_2 \in exp(\h_0 \cap \p), l\in L, z\in Z$, we define $\Phi(z)(x_2 l):=K_T(e,x_2 l)(z)= \tau(l^{-1})K_T(e,x_2 )(z)$. Hence, $\Phi(z)(\cdot)\in \Gamma^\infty(H_0 \times_\tau W)$. Next, the functional equation, \ref{eq:forK_T},  for the kernel $K_T$ yields  $\Phi$ satisfies \ref{eq:lk_t}, and,  owing to $R_X (K_T(h, \cdot)z)=0$ for each $h\in H, X \in \p^-$ we obtain $\Phi$ is "holomorphic", that is, $R_X (\Phi (z)(\cdot))=0$ for each $ X \in \p_{\h_0}^-$. Therefore, $\Phi \in Hom_L(Z, \mathcal O (H_0 \times_\tau W))$. To follow, we verify $\Phi(z)(\cdot)$ belongs to $L^2(H_0 \times_\tau W)$.  Owing to the representation $H_{hol}^2(G\times_\tau W)$ is $H$-admissible,  we may apply  the result of T. Kobayashi \cite{Kob2}  "In a $H$-admissible representation, every $L$-finite vector is $K$-finite" and we obtain,  $K_T(e,\cdot)z \in H_{hol}^2(G\times_\tau W)[H_{hol}^2(H\times_\sigma Z)][Z]\subset H_{hol}^2(G\times_\tau W)_{K-fin}$. A result due to    Schmid, Knapp-Wallach \cite{Sch} \cite[Cor. 9.6]{KW}       shows $H_{hol}^2(G\times_\tau W)_{K-fin}= H^2(G\times_\tau W)_{K-fin}$ and $H_{hol}^2(G\times_\tau W) \subset  H^2(G\times_\tau W)$. Finally,   $r_0 $ maps $H^2(G\times_\tau W)$ into $ L^2(H_0 \times_\tau W)$ (see \cite{OV}), and,   $\Phi(z)(x_2l)=\tau(l^{-1})r_0(K_T(e,\cdot)z)(x_2) $, whence   we have verified $\Phi(z)(\cdot)\in H_{hol}^2(H_0 \times_\tau W) $.

 \smallskip
   It readily follows: $T_\Phi(g)$ is "holomorphic"  for each $g \in H_{hol}^2(H\times_\sigma Z)$.

 \smallskip
  Since, via $E_{c_\sigma}$, $H_{hol}^2(H\times_\sigma Z)$ is isomorphic to the kernel of an elliptic operator, we have that $L^2$-convergence in $H_{hol}^2(H\times_\sigma Z)$ implies convergence in the natural topology in $\mathcal O(H\times_\sigma Z)$,    we  conclude
  $T_\Phi :H_{hol}^2(H\times_\sigma Z)\rightarrow \mathcal O(G\times_\tau W)$ is a continuous linear map.  Moreover,  since for $g \in H_{hol}^2(H\times_\sigma Z)$ the equality $g(h)=\int_H K_\sigma (y,h)g(y) dy, h\in H$  holds, we have $T_\Phi$ is a kernel linear map represented by the kernel \begin{equation*}K_{T_\Phi}(h, x_1x_2k)(\cdot):=\tau(k^{-1})\Phi ( K_\sigma (h,x_1)(\cdot))(x_2). \end{equation*}

  Owing to $T_\Phi(K_\sigma (\cdot,e)^\star)(z))= K_{T_\Phi}(e, \cdot)(z)$, as in Proposition~\ref{prop:propertiesholo}(5), we obtain $K_{T_\Phi}(e, \cdot)(z)$ is a $L$-finite vector. Thus, $E_{c_\tau}(K_{T_\Phi}(e,\cdot)z)$ is a $L$-finite vector in the $H$-admissible representation $(L_\cdot^\tau , \mathcal O(\mathcal D,W))$, Kob implies $E_{c_\tau}(K_T{T_\Phi}(e,\cdot))$ is a $K$-finite vector, now,  we recall that a $K$-finite holomorphic function is polynomial, thus $E_{c_\tau}(K_{T_\Phi}(e,\cdot))\in \mathcal O (\mathcal D,W)\cap L_\tau^2(\mathcal D,W)$ and hence, $K_{T_\Phi}(e,\cdot)z$  as well as $L_h(K_{T_\Phi}(e,\cdot))=K_{T_\Phi}(h,\cdot)$ belongs to $H_{hol}^2(G\times_\tau W)$. Finally, the equality $ K_{T_\Phi}(e,\cdot)z=T_\Phi(K_\sigma (\cdot,e)^\star z)$ yields $ K_{T_\Phi}(e,\cdot)\in H_{hol}^2(G\times_\tau W)[H_{hol}^2(H\times_\sigma Z)[Z] \subset \mathcal L_{W,H}$.
   Also,   $K_T(e,\cdot)(z)\in \mathcal L_{W,H}$. By construction, $r_0(K_T(e,\cdot)z)=r_0(K_{T_\Phi}(e,\cdot)z)$, whence,    due that $r_0$ is injective in $\mathcal L_{W,H}$ (see \ref{sec:r0injec2}) we obtain that both kernels are equal and $T=T_\Phi$.
  \end{proof}

 \begin{rmk} The converse to Theorem~\ref{prop:kernforhol} holds. In fact, given $\Phi \in Hom_L(Z, H_{hol}^2(H_0 \times_\tau W)$, the duality theorem \ref{prop:direciso} shows there exists $T_1 \in  Hom_H(H_{hol}^2(H\times_\sigma Z), H_{hol}^2(G\times_\tau W))$ so that $r_0(K_{T_1}(e,\cdot)z)=\Phi(z)$.
 \end{rmk}

 \begin{cor}\label{prop:splitform} The kernel $K_T$ of each holographic operator $ T \in Hom_H(H_{hol}^2(H\times_\sigma Z), H_{hol}^2(G\times_\tau W))$, decomposes as the composition "separation of variables formula" \begin{multline*}K_T(h, x_1x_2k)(z) =\tau(k^{-1})K_T(e,x_2)(K_\sigma (h,x_1)z),\\ h\in H, x_1 \in exp(\h \cap \p), x_2 \in exp(\h_0 \cap \p), k\in K. \end{multline*}.
 \end{cor}
 Theorem~\ref{prop:kernforhol}  and its corollary generalizes     \cite[example 10.1]{OV2} for the case $G=SU(n,1), H=S(U(n-1,1)\times U(1)), H_0=S(U(n-1)\times U(1,1))$.

 \bigskip
 To follow, we rewrite Theorem~\ref{prop:direciso} in the symmetric space model, later on, after a change of the  coordinates $(x_1, x_2) $ to   holomorphic coordinates we obtain another formula for the kernel of a holographic operator.

 We recall that for each holographic operator $T: H_{hol}^2(H\times_\sigma Z)\rightarrow H_{hol}^2(G\times_\tau W)$ we have associated a unique holographic operator   $T^c : H_{c_\sigma}^c (H/L, Z) \rightarrow H_{c_\tau}^c (G/K, W)$ via the composition $T_c =E_{c_\tau} T E_{c_\sigma}^{-1}$ and the respective kernels are related by the equality~\ref{eq:ttc} $$ K_{T^c}^c(hL,xK)=c_\tau (x,e)K_T (h,x) c_\sigma (h,e)^\star, \, \forall h \in H, x \in G. \eqno(\natural)$$
 Hence, the formula for $K_T $ obtained in Corollary~\ref{prop:splitform} and \ref{prop:propertiesholo} (5) yields \begin{cor}
 \begin{equation*} K_{T^c}^c( hL, x_1x_2K)=\tau( (x_1exp{(x_2)_+})_0)K_{T^c}^c (o,(x_1)_0 x_2K)    K_\sigma^c(hL, x_1K). \end{equation*}
 \end{cor}
  Here, $ x_1 \in exp(\h \cap \p), x_2 \in exp(\h_0 \cap \p), G\ni  a=exp(a_+)\,(a)_0\,\, exp(a_-)$, $a_\pm \in \p^\pm, (a)_0 \in K_\C$.

  \smallskip
  In fact, the equality $\natural $, $a\cdot o=a_+$, $c_\tau (x_1x_2,o)=c_\tau (x_1, x_2\cdot o)c_\tau (x_2,o)$ \\ $= \tau((x_1 exp((x_2)_+)))_0)\tau ((x_2)_0)$, $c_\sigma (b,o)=\sigma ((b)_0)  $  and \ref{eq:ktaucktau} let us obtain

 \smallskip
  $K_{T^c}^c( hL, x_1x_2K)z$

   $= \tau((x_1 exp((x_2)_+)))_0)\tau ((x_2)_0) K_{T}(e,x_2)(K_\sigma (h,x_1)\sigma ((h)_0)^\star z$

  $=\tau((x_1 exp((x_2)_+)))_0)  K_{T^c}^c(e,x_2K)\sigma( (x_1)_0)^{-1} K_\sigma^c (hL,x_1K)z$.

  \medskip

\cite[Example 10.1]{OV2}  for $(SU(n,1), S(U(n-1,1)\times U(1))$ and  holomorphic scalar representations in the bounded symmetric space realization, shows the factor

  \phantom{xxxxxxxxxx} $\tau( (x_1exp{(x_2)_+})_0)\, K_{T^c}^c (o,x_2K)\, \sigma ((x_1)_0^{-1})$

   \noindent
   does not depends on $x_1, "w_1" \in \p_\h^+$.
  We do not know,  in the bounded symmetric realization of holomorphic Discrete Series,  when "separation of variables via $K_\tau^c $ " does   hold. On the positive side, in the $H_{hol}$ realization, Corollary~\ref{prop:splitform} gives:

 \begin{cor}\label{cor:symmsplit} Let $S:  H_{hol}^2(G\times_\tau W) \rightarrow H_{hol}^2(H\times_\sigma Z)  $ be a symmetry breaking operator. Then, $$K_S(x_1x_2 k, h)=  K_S(x_2,e)K_\sigma(x_1,h) \tau (k). $$
 \end{cor}

   \begin{examp} $G=SU(1,1)\times SU(1,1)$, $\sigma(x,y)=(y,x)$, $H=\{ (x,x): x\in SU(1,1)\}$,  $K=T=\{(diag(e^{i\phi}, e^{-i\phi}), diag(e^{i\psi}, e^{-i\psi})) \}$,

     $G^{\sigma \theta}=H_0=\{ (x,x^{*-1}): x\in SU(1,1)\}$, $\h_0=LieH_0=\{ (x,-x^{*}): x\in \mathfrak{su}(1,1)\}=(1\times \theta_{\mathfrak{su}(1,1)})H$.
 \begin{multline*}  exp(\p \cap \h)   = \{ x_1(a):=   \big(cosh(\vert a \vert) I \\ +\frac{sinh(\vert a \vert)}{\vert a \vert}(\begin{smallmatrix}   0 & a \\ \bar a & 0 \end{smallmatrix}), \, cosh(\vert a \vert) I+\frac{sinh(\vert a \vert)}{\vert a \vert}(\begin{smallmatrix}   0 & a \\ \bar a & 0 \end{smallmatrix})\big),  a \in \C \}. \end{multline*}
 \begin{multline*} exp (\p \cap \mathfrak h_0)   = \{  x_2(a):=  \big(cosh(\vert a \vert) I \\ +\frac{sinh(\vert a \vert)}{\vert a \vert}(\begin{smallmatrix}   0 & a \\ \bar a & 0 \end{smallmatrix}), \, cosh(\vert a \vert) I+\frac{sinh(\vert a \vert)}{\vert a \vert}(\begin{smallmatrix}   0 & -a \\ -\bar a & 0 \end{smallmatrix})\big),  a \in \C \}. \end{multline*}

 The generalized Cartan decomposition for $G$ reads:

 \phantom{xxxxxx} $G\ni x=x_1(a) x_2(b) k, \,\, a ,b \in \C, k\in K $.

   We set   \phantom{xxxxxxx} $\tau_\lambda (exp(\phi \left ( \begin{smallmatrix} i&0 \\ 0 &-i \end{smallmatrix} \right)))=e^{i\lambda \phi}, \lambda \geq 2$.

  Let $v_s(\begin{psmallmatrix} \alpha &\beta \\ \bar\beta& \bar\alpha \end{psmallmatrix}):=(\lambda)_s \bar\alpha^{-\lambda -s} \beta^s$, then,  $v_s \in H_{hol}^2(SU(1,1), \tau_{\lambda})[\tau_{\lambda+2s}]$.  An orthonormal basis for $H_{hol}^2(SU(1,1),\tau_\lambda)$  is $\{v_s, s=0,1,\dots \}$.  For $H_{hol}^2(SU(1,1),  \tau_{\lambda'})$ we denote the corresponding basis for $v_t^\prime, t=0,1,\dots$.

 A computation in the unit disc,  for $\lambda, \lambda' \geq 2$, yields

 $H_{hol}^2(SU(1,1), \tau_\lambda)\otimes  H_{hol}^2(SU(1,1), \tau_{\lambda'})[\tau_{\lambda +\lambda'+2n}]\cap \mathcal L_{\C, SU(1,1)} $

\phantom{cccccccccccccc} $=\C \, \Phi(1)(\cdot)$, where, for $z\in \C$,

  \phantom{xxxxxx} $\Phi(z)(\cdot)=  z\sum_{0\leq s \leq n } (-1)^s  \binom{n}{s} \frac{(\lambda -1)!} {(\lambda-1 +s)!} \,\,\frac{(\lambda^\prime-1)!}{(\lambda^\prime -1+n-s)!} v_s\otimes v_{n-s}^\prime(\cdot) $ .

  Therefore, the associated holographic operator

  \noindent
  $T_\Phi :  H_{hol}^2(SU(1,1), \tau_{\lambda+\lambda'+2n}) \rightarrow H_{hol}^2(SU(1,1), \tau_\lambda)\otimes  H_{hol}^2(SU(1,1), \tau_{\lambda'})$  is the operator
 that maps $g\in H_{hol}^2(SU(1,1), \tau_{\lambda+\lambda'+2n})$ to
  \begin{multline*}
 T_\Phi(g )( x_1(a) x_2(b) (k_1, k_2)) \\ =\tau_\lambda \otimes \tau_{\lambda^\prime}(k_1,k_2)\big(\Phi(g(x_1(a))\big)(x_2(b))\\
 = \tau_{\lambda}(k_1^{-1}) \tau_{\lambda'}(k_2^{-1})\,\,
  g(x_1(a))  \\
 \times \sum_{0\leq s \leq n } (-1)^s  \binom{n}{s} \frac{(\lambda -1)!} {(\lambda-1 +s)!} \,\,\frac{(\lambda^\prime-1)!}{(\lambda^\prime -1+n-s)!} v_s(x_2(b))  v_{n-s}^\prime (x_2(b)^{*-1}). \end{multline*}

   \end{examp}

 \begin{rmk} For an arbitrary symmetric pair $(G,H)$, arbitrary $H^2(H,\sigma)$,  $H^2(G,\tau)$,  and, for $  \Phi \in  Hom_L(Z, \mathbf H^2(H_0 \times_\tau W))$, the linear map $T_\Phi$  defines a  $H$-map,  $T_\Phi \in Hom_H (   H^2(H\times_\sigma Z) , C^\infty (G \times_\tau W)) $. If we could show $D_{Schmid}(K_{T_\Phi}(h,\cdot))=0$ and $\int_G \Vert T_\Phi(g) \Vert^2 dg <\infty$ for all $g \in H_{hol}^2(H,\sigma)$ we would obtain holographic operators. By means of the analysis of leading exponents of a representation (see \cite[Chap IV]{Wa1}) we are able to show that $T_\Phi(g)$ is square integrable for Harish-Chandra parameters far away from the walls. The equality $D_{Schmid}(K_{T_\Phi}(h,\cdot))=0$ would follows if we knew a Hartog's Theorem for the Schmid operator. We refer to Hartog's Theorem as the Theorem that shows: a function is holomorphic if and only if it is holomorphic in each variable.
 \end{rmk}
\subsubsection{A result of  Kitagawa on holographic operators}  Recently, some remarkable new results have been obtained by Kitagawa, as follows: Let $(L_\cdot^G, V^G:=H^2(G,\tau))$ be a $H$-admissible Discrete Series representation. Let $V^H:=H^2(H,\sigma)$ be an irreducible factor of $res_H(V^G)$.

The $H$-admissibility hypothesis yields:\\
For any intertwining map $ T:(V^G)^{\infty}[V^H] \rightarrow  (V^G)^{\infty}$, in particular,  for each $R\in \mathcal U(\g)^H$, the either the map $T$ or  $L_R^G :(V^G)^{\infty}[V^H] \rightarrow  (V^G)^{\infty}$,   extents to a continuous linear endomorphism for $V^G[V^H]$.

In fact, we write $V^G[V^H]=M_1 \oplus \cdots \oplus M_k$, where $M_i$ are irreducible $H$-factors, thus $L_R^G$ maps the $\mathcal U(\h)$-irreducible representation $(M_i)_{L-fin}$ into the unitary representation $V^G[V^H]$, we now apply \cite[Lemma 8.6.7]{Wal} and obtain $T$ or $L_R^G$ extents to a continuous linear map from $M_i$ into $V^G[V^H]$, whence, the claim follows.

A consequence of this that for any $T\in Hom_H(V^H, V^G)$ and for any $R\in \mathcal U(\g)^H$, the composition map $L_R^G\, T$ from $(V^H)_{L-fin} \rightarrow V^G$  extends to a holographic operator.  Whence,  $\mathcal U(\g)^H$ also acts in $Hom_H(V^H,V^G)$ by the rule $R\cdot T=L_R^G \, T$.  We note that we have verified the inclusion $Hom_H(V^H,V^G)\subseteq Hom_H((V^H)^{\infty},(V^G)^{H-\infty})$ becomes an equality.  In \cite[Theorem 5.26]{Ki}, appealing to a particular Zuckerman functor's realization of the space of $K$-finite vectors of a Discrete Series representation,   it is shown that the action of $\mathcal U(\g)^H$  in $Hom_H(V^H,V^G) $  is  irreducible. Thus, we obtained the observation: once we know one nonzero holographic operator $T$ from $V^H$ into $V^G$, all the others have the shape $L_R^G \, T, R\in \mathcal U(\g)^H$.

   \section{Further perspectives} The aim of this section is to present an overview of some of the results that will be shown in Part II.  A major goal of branching laws is  to understand the
  structure of symmetry   breaking (holographic) operators for a general pair $(G,H)$
  and $(\pi, V)$  a $H$-admissible Discrete Series. For this, we present some considerations on the structure of symmetry breaking operators. Whence,  we fix  a symmetry breaking operator $S:H^2(G,\tau)\mapsto H^2(H,\sigma)$, represented by the kernel $K_S :G\times H \rightarrow Hom_\C( W, Z)$ and recall  the subspace $Z_S:=Image(Z\ni z \mapsto K_{S^*}(e, \cdot) (z)=K_S(\cdot, e)^\star z \in H^2(G,\tau))$ is a $L$-irreducible subspace contained in $\mathcal L_{W,H}$. Thus, either $Z_S \cap \mathcal U(\h_0)W=\{0\}$ or $Z_S \subset \mathcal U(\h_0)W$. In \cite{OV3}, it is shown that in the later case $S$ is represented by a normal derivative differential operator, whereas in the former case $S$ is represented by a differential operator that never will be a normal derivative differential operator. To be more precise, we write

    $ (\ddag)\,\,\,  Hom_H(V,V_\sigma) =\{ S: Z_S \subset \mathcal U(\h_0)W\}\cup \{ S: Z_S \cap \mathcal U(\h_0)W=\{0\}\}$.

  After ignoring the zero operator,  this is a disjoint union,  the first subset is the one that contains the symmetry breaking operators that are represented by normal derivatives operators and the second subset is its complement.   Roughly speaking,    $ \mathcal L_{W,H}\cap \mathcal U(\h_0)W  $ measures the "quantity" of symmetry breaking operators   represented by normal derivatives operators, whereas, $ \mathcal L_{W,H} \backslash (\mathcal L_{W,H}\cap \mathcal U(\h_0)W)$,  "measures"  the totality  of "non normal" derivative symmetry breaking operators. We do not know if for given $\sigma, \tau$ both subsets in $ (\ddag)$ might be nonempty. That is, we do not know whether every nonzero element in $Hom_H(V_\tau, V_\sigma)$  is represented by a normal differential operator or not

  Obviously,  if  $H^2(H,\sigma)$ has multiplicity one in $H^2(G,\tau)$, the symmetry breaking operators in $Hom_H(V,V_\sigma)$ are either normal derivatives differential operators or not.

Henceforth,  our hypothesis is: $H/L \rightarrow G/K$ is a holomorphic immersion and $(\pi,V)$ is $H$-admissible holomorphic representation of lowest $K$-type $(\tau,W)$.

  For a scalar holomorphic  representation $\tau$ and $(\g,\h)$ a real rank one pair,  in \cite[Table 4.1, Theorem 5.1]{KP2} it is shown: For the pairs $(\mathfrak{su}(m,n), \mathfrak s(\mathfrak{u} (m,n-1) + \mathfrak u(1)))$, $(\mathfrak{so}(2m,2), \mathfrak u(m,1))$,    $(\mathfrak{so}^\star(2n),  \mathfrak{so}(2) +\mathfrak{so}^\star(2n-2)) $   every      symmetry breaking operators is represented by normal derivative operators. Whereas, for the pairs $ (\mathfrak{so} (m,2), \mathfrak{so}(m-1, 2))$,  \\ $(\mathfrak{su}(m,1)\oplus \mathfrak{su}(m,1),    \mathfrak{su} (m,1) )$,  $(\mathfrak{sp} (m,1), \mathfrak{sp}(m+1,\mathbb R)\oplus \mathfrak{sp}(1, \mathbb R))$ there exists   symmetry breaking operators not represented by normal derivative operator's. We can provide a new  proof of their result.

  \smallskip

   Due that the Lie algebra $\p^-$  is anabelian Lie algebra, the symmetrization from $S(\p^-)$ onto $\mathcal U(\p^-)$ becomes an associative algebra isomorphism. Whence, after we recall \ref{eq:1},\, $\mathcal U(\p^-)\otimes W$ may be written as a  graded vector space, the $n^{th}$ subspace been $\mathcal V^{(n)}:=lin.span_\C \{L_{x_1\cdots x_n}^\tau (w)$, $ x_j \in \p^- \backslash \{0\}, w\in W\}$.

   We will compute    examples so that the following inclusions  might be proper,
  $ \{0\}\subset   \mathcal L_{W,H} \cap \mathcal V^{(1)} \cap \mathcal U(\h_0)W \subset  \mathcal L_{W,H}\cap \mathcal V^{(1)}  $. Hence, since for a symmetry breaking operator $S$ so that $K_S(\cdot,o)^\star (Z) \subset \mathcal L_{W,H}\cap \mathcal V^{(1)}$ we know it is always represented by a differential operator, and   $K_S(\cdot,o)^\star (Z) \subset \mathcal U(\h_0)W \cap \mathcal V^{(1)}$ is equivalent to being a normal derivative operator. We have that there are symmetry breaking differential operators which are not represented by normal derivative even though there are non trivial normal derivative symmetry breaking operators.  The next Proposition analyzes the question of when every first order symmetry breaking operator is represented by a normal derivative operator.  \begin{prop} When $\g$ is simple and  is not isomorphic  to $\mathfrak{su}(m,n)$, $m\geq 2, n\geq 2$, then,  $\mathcal U(\h_0)W \cap \mathcal L_{W,H} \cap  \mathcal V^{(1)} =  \mathcal L_{W,H}\cap \mathcal V^{(1)}$ if and only if $\tau$ is a one dimensional representation. \end{prop} We also have a similar equivalence for $\g\equiv \mathfrak{su}(m,n)$, \, $m\geq 2, n\geq 2$.

  \medskip
  The next statement, shows that under certain hypothesis we find first order symmetry breaking operators represented by normal derivatives as well as that there exists first order symmetry breaking operators which are not represented by normal derivative. \begin{prop} For a pair $(\g,\h)$ so that both $\h,\h_0$ are isomorphic to the product of one noncompact simple Lie algebra times a compact Lie algebra, and, $\tau$ is not a scalar representation, the representation of $L$ in
  $\mathcal U(\h_0)W \cap\mathcal L_{W,H} \cap    \mathcal V^{(1)}$ is irreducible and not a scalar representation.  Moreover, $\mathcal U(\h_0)W \cap \mathcal L_{W,H} \cap    \mathcal V^{(1)}$ is a proper subspace of $\mathcal L_{W,H} \cap   \mathcal V^{(1)}$. \end{prop}

  \medskip

  The following fact shows that, sometimes,  every second order symmetry breaking operator is represented via normal derivative.
   For a scalar representation $(\tau,W)$, we claim: \begin{prop} $\mathcal U(\h_0)W \cap \mathcal L_{W,H} \cap    \mathcal V^{(2)}  =  \mathcal L_{W,H}\cap \mathcal V^{(2)}$ if and only if $[[\p_{\h_0}^+, \p_{\h}^-],\p_{\h_0}^+]=\{0\}$.  \end{prop} It  follows, via the Poincare-Birkoff-Witt Theorem,  that the equalities
$ \mathcal U(\h_0)W \cap \mathcal V^{(i)}=  \mathcal L_{W,H} \cap   \mathcal V^{(i)}, i=1,2$ yields the equality $\mathcal L_{W,H} =  \mathcal U(\h_0)W$.

\subsection{Examples and comments on Branching Laws}

The concrete problems of Branching Laws entails the use of special functions, Fourier, Laplace y other transforms as we can learn from the work of Kobayashi-Pevzner, Frahm-Oshima and  many other researchers. In the case of working of Branching Laws for Discrete Series representations appears new mathematical objects, like reproducing kernel spaces, Schmid operators as well as the classical Casimir operator. In this note, we essentially have analyzed holomorphic Discrete Series representations, however, there are non-holomorphic Discrete Series with admissible restriction to appropriate subgroups, for references, and examples, we suggest to browse the work of Gross-Wallach,  Kobayashi-Oshima, Orsted-Vargas, Genkai Zhang and many other authors. Many open problems remain, for example, in \cite{OV2}, for a quaternionic Discrete Series representation $\pi^{Sp(1,p+1)}$,  we derive an abstract decomposition for $res_{Sp(1)\times Sp(1,p)}(\pi^{Sp(1,p+1)})$, however, we do not explicitly present such a decomposition. That is, we do not provide equations for isotropic subspaces $\pi^{Sp(1,p+1)}[\pi^{Sp(1)\times Sp(1,p)}]$ as well as for the symmetry (holographic) operators. Other open  problem is to explicit the continuous spectrum for the restriction of a Discrete Series representation, important results on the subsect are found in the work Harris-He-Olafsson and many other authors.

To finish this section, we explicit two examples of admissible restriction. For details, \cite{OV3}. Here, the symmetric pair $(G,H)$  is so that the corresponding pair of Lie algebras is:
$ (\mathfrak{e}_{6(2)}, \mathfrak f_{4(4)})  $.  This is not a holomorphic pair. However, there exists a Borel-de Siebenthal system $\Psi_{BS}$ of positive roots in $\Phi(\mathfrak e_6, \mathfrak t)$ so that a Discrete Series for $E_{6(2)}$ whose Harish-Chandra parameter is dominant with respect to $\Psi_{BS}$ has an admissible restriction to $H\equiv F_{4(4)}$.  In particular, this holds for the quaternionic, that is,  the $SU_2(\alpha_{max})$-finite,  representations $H^2( E_{6(2)}, \pi_{n\frac{\alpha_{max}}{2}+\rho_{SU(6)}}^{SU_2(\alpha_{max})\times SU(6) }), n=1,2,\cdots $.  In order to be more explicit, we fix a compact Cartan subgroup $T\subset K\equiv SU_2(\alpha_{max})\times SU(6)$ so that $U:=T\cap H$ is a compact Cartan subgroup of $ L=K\cap H \equiv SU_2(\alpha_{max})\times Sp(3) $. Then, the    quaternionic and Borel de Siebenthal positive root system $\Psi_{BS}$ for $\Phi(\mathfrak e_6, \mathfrak t)$ is,  after   we write the simple roots   as in Bourbaki,     compact simple roots are $\alpha_1,  \alpha_3, \alpha_4, \alpha_5, \alpha_6$ (They determinate the $A_5$-Dynkin sub-diagram)  and $\alpha_2$ is   noncompact. $\alpha_2$ is adjacent to $-\alpha_{max}$ and to $\alpha_4$.

     The automorphism  $\sigma$  of $\mathfrak g$  acts on the simple roots as follows $$\sigma (\alpha_2)=\alpha_2, \,\,  \sigma (\alpha_1)=\alpha_6, \,\,  \sigma( \alpha_3 )=\alpha_5, \,\, \sigma (\alpha_4)=\alpha_4.$$   The associated pair is $(\mathfrak e_{6(2)}, \mathfrak{sp}(1,3))$. Let $q_\mathfrak u$ denote the restriction map from $\mathfrak t^\star$ into $\mathfrak u^\star$. Then, for $\lambda$ dominant with respect to $\Psi_{BS}$,  the simple roots for $\Psi_{H,\lambda}=\Psi_{\mathfrak f_{4(4)},\lambda}$,  respectively $\Psi_{\mathfrak{sp}(1,3), \lambda},$   are:
\begin{center}
   $\alpha_2, \,\, \alpha_4, \,\,  q_\mathfrak u(\alpha_3)=q_\mathfrak u (\alpha_5), \,\,  q_\mathfrak u(\alpha_1)=q_\mathfrak u (\alpha_6). $

 $  \beta_1=q_\mathfrak u(\alpha_2 + \alpha_4 +\alpha_5)=q_\mathfrak u (\alpha_2 + \alpha_4 +\alpha_3), \,\, \beta_2=q_\mathfrak u(\alpha_1)=q_\mathfrak u (\alpha_6)$,  \\ $\,\,   \beta_3=q_\mathfrak u(\alpha_3)=q_\mathfrak u (\alpha_5), \,\, \beta_4= \alpha_4.   $
\end{center}
Both systems of positive roots,  $\Psi_{\mathfrak f_{4(4)},\lambda}$,    $\Psi_{\mathfrak{sp}(1,3), \lambda} $ satisfies the Borel-de Siebenthal property and they are quaternionic.
 The fundamental weight $\tilde{\Lambda}_1$ associated to $\beta_1$ is equal to $\frac{1}{2} \beta_{max}$. Hence,

 $ \tilde\Lambda_1=\beta_1+\beta_2+  \beta_3 + \frac12 \beta_4=\alpha_2+\frac{3}{2}\alpha_4 + \alpha_3+\alpha_5+\frac{1}{2}(\alpha_1+\alpha_6) $.

$\rho_c= \rho_{SU(6)}+ \rho_{SU_2(\alpha_{max})}=\frac{5}{2}\alpha_1+\frac{8}{2} \alpha_3+\frac{9}{2}  \alpha_4+\frac{8}{2}  \alpha_5+\frac{5}{2}  \alpha_6 +\frac{1}{2} \alpha_{max}$,

 $\rho_n^{\Psi_{BS}}= \frac{11}{2}\alpha_{max}$, $\rho_n^{\Psi_{\mathfrak{sp}(1,3), \lambda}}= \frac{3}{2}\alpha_{max}$, $\rho_{Sp(3)}=3 \beta_2 +5\beta_3+3\beta_4$.

 The highest weight of the lowest $K$-type for Discrete Series of Harish-Chandra parameter $n\frac{\alpha_{max}}{2}+\rho_{SU(6)}$ for $E_{6(2)}, n\geq 1$ is $$\nu_n:=n\frac{\alpha_{max}}{2} +\rho_{SU(6)} +\rho_{n}^{\Psi_{BS}}-\rho_c =(n+10)\frac{\alpha_{max}}{2}.$$

Thus,     the Duality Theorem,   applied to the restriction of the quaternionic representation

\smallskip
\phantom{xxxxxxxxxxxxxx}$H^2( E_{6(2)}, \pi_{\nu_n }^{SU_2(\alpha_{max}) \times SU(6)})$

to $F_{4(4)}$, yields

\begin{multline*} res_{ F_{4(4)}}\big(H^2(E_{6(2)}, \pi_{\nu_n }^{SU_2(\alpha_{max}) \times SU(6)})\big) \\ =\bigoplus_{ m\geq 0 }\,\, H^2(F_{4(4)}, \pi_{ \sigma_{n,m}}^{SU_2(\alpha_{max}) \times Sp(3)}) .
 \end{multline*}

Here, the highest weight $\sigma_{n,m}$ is:  $\sigma_{n,m}=(n+11+m)\frac{\alpha_{max}}{2} +m\tilde\Lambda_1 $.

However, this is in some sense an abstract decomposition, owing to we do not  provide either  the equations  or a description of each isotopic component $$ H^2(E_{6(2)}, \pi_{n\frac{\alpha_{max}}{2}+\rho_{SU(6)} }^{SU_2(\alpha_{max})\times SU(6)})\bigl[ H^2(F_{4(4)}, \pi_{(n+11+m)\frac{\alpha_{max}}{2} +m\tilde\Lambda_1 }^{SU_2(\alpha_{max}) \times Sp(3)})\bigl].$$ Nevertheless,  in \cite[Proposition 6.8]{OV2} we have obtained a formula for the kernel that represents the orthogonal projector onto a given isotypic component, for the precise statement see Remark 2.2. We note that the Harish-Chandra parameter of the irreducible $F_{4(4)}$-factors are dominant with respect to a Borel-de Siebenthal system of positive roots. Moreover, each $F_{4(4)}$-irreducible factor is a generalized quaternionic representation. This ends the first example.

 \medskip

 The second example is on the pair $(\mathfrak{so}(2m,2), \mathfrak{so}(2m,1))$.   We restrict from  $Spin(2m,2), m \geq 2$, to $Spin(2m,1)$. We notice the   isomorphism between $(Spin(4,2), Spin(4,1))$  and  the pair $(SU(2,2),Sp(1,1))$. In this setting,   $K=Spin(2m)\times Z_K$, $L=Spin(2m)$, $ Z_K\equiv \mathbb T$.  Obviously, we may conclude that any irreducible representation of $K$ is irreducible when restricted to $L$. In this case $H_0 \equiv Spin(2m,1)$, and (for $m=2$,  $H_0 \equiv Sp(1,1)$). We always have: the representation $\mathbf H^2(H_0,\tau)$ is irreducible. Therefore, the duality theorem together with that  any irreducible representation for $    Spin(2m,1) $ is $L=Spin(2m)$-multiplicity free   \cite[page 11]{Th}, let us to obtain:

 \smallskip
  {\it Any   Discrete Series representation for $  Spin(2m,2)  $, with a admissible restriction to  $    Spin(2m,1) $,  is a  multiplicity free representation.}
  \smallskip

We fix a maximal torus $T$ for $K$, so that $U:=L\cap T$ is a maximal torus for $L$. Then, there exists a orthogonal basis $\epsilon_1, \dots, \epsilon_m, \delta$ for $i\mathfrak t^\star$ so that $\z_K^\star = \C \delta$, and $\Phi(\mathfrak{so}(2m,2),\t):=\{ \pm(\epsilon_k \pm \epsilon_s), 1 \leq k < s \leq m \}\cup  \{ \pm (\epsilon_j \pm \delta), 1 \leq j \leq m \}.$
  We consider the systems  of positive roots in $\Phi(\mathfrak{so}(2m,2),\t)$ defined as follows: $\Psi_-=S_{\epsilon_m -\delta}S_{\epsilon_m +\delta}\Psi_+$ and

 \phantom{xxxxx} $\Psi_+= \{ \epsilon_k \pm \epsilon_s, 1 \leq k < s \leq m, (\epsilon_j \pm \delta), 1 \leq j \leq  m       \}$.

 The systems $\Psi_\pm$ are not Borel-de Siebenthal.

 For $ Spin(2m,2), m\geq 3 $    in \cite[Table 2 ]{Va}, \cite{KO} it is verified that any Discrete Series of Harish-Chandra parameter $ \lambda$   dominant with respect to one of the systems $\Psi_{\pm}$    has  admissible restriction to $Spin(2m,1)$ and no other Discrete Series  representation    has admissible restriction to $Spin(2m,1)$.

 We compute, $\rho_n^{\Psi_\pm} =m(\epsilon_1 +\cdots +\epsilon_{m-1} \pm \epsilon_m)$, $\rho_c=(m-1)\epsilon_1 +\cdots -(m-1)\epsilon_m$. Let  $q_\u$ denote the restriction map from $\t^\star$ onto $\u^\star$. For $\lambda \in  i\t^\star$ dominant integral for one of the systems  $\Psi_\pm$, the highest weight Harish-Chandra parameter (infinitesimal character) of the lowest $K$-type for the Discrete Series of $Spin(2m,1)$ attached to $\lambda$ is $
  \lambda +\rho_n^{\Psi_\pm}$. Thus, \\ \phantom{xxxxxxxx}   $(\tau, W)=(\tau_{\lambda +\rho_n^{\Psi_\pm}}^{Spin(2m)\times SO(2)}, V_{\lambda +\rho_n^{\Psi_\pm}}^{Spin(2m)\times SO(2)})$.

  The restriction of $(\tau, W)$ to $L$ is the irreducible the representation \\ \phantom{xxxxxxxxxxxxxx} $(\tau_{q_\u(\lambda +\rho_n^{\Psi_\pm})}^{Spin(2m)},(V_{q_\u(\lambda +\rho_n^{\Psi_\pm})}^{Spin(2m)})$.

  Let $Spec_L(H^2(Spin(2m,1), res_L((\tau,W)))) \subset i\u^\star$ denote the totality  of  the infinitesimal character, dominant with respect to $\Psi_\pm \cap \Phi_c \cap i\u^\star$,  of the irreducible $L$-factors of  $H^2(Spin(2m,1), res_L((\tau,W)))$. In \cite{Th}, we find an algorithm to compute the set $Spec_L(H^2(Spin(2m,1),$ $ res_L((\tau,W))))$. Then, the duality theorem yields

 \begin{multline*} res_{Spin(2m,1)}\big(H^2(Spin(2m,2), \tau_{\lambda +\rho_n^{\Psi_\pm}}^{Spin(2m)\times SO(2)})\big)\\
   =\bigoplus_{\sigma \in Spec_L(H^2(Spin(2m,1), res_L((\tau,W))))}\, $ $H^2(Spin(2m,1), \tau_\sigma^{Spin(2m)}). \end{multline*}

The above formula provides a description of the restriction, however, we have not been able to compute the isotypic component

$$H^2(Spin(2m,2), \tau_{\lambda +\rho_n^{\Psi_\pm}}^{Spin(2m)\times SO(2)})[H^2(Spin(2m,1), \tau_\sigma^{Spin(2m)})].$$

as well as, either holographic or symmetry breaking operators. Nevertheless,  in \cite[Proposition 6.8]{OV2} we have obtained a formula for the kernel that represents the orthogonal projector onto a given isotypic component, for the precise statement see Remark 2.2.
 \section{Partial list of symbols and definitions}
 \noindent
 -$(\tau ,W),$ $(\sigma, Z)$, $L^2(G \times_\tau W), L^2(H\times_\sigma Z)$, $\,H^2(G,\tau)=V_\lambda=V_\lambda^G $, $H^2(H,\sigma)=V_\mu^H, L_\cdot^H=\pi_\mu^H, \tau=\pi_\nu^K.  $ (cf. Section 1).\\
 -$V_\tau$, $V_\sigma$, $L_\cdot^G =L_\cdot^G$, $d_\lambda=d(L_\cdot^G)$ $L_\cdot^G,$   $   K_\tau,  K_\sigma,  $  (cf. Section 2). \\
 -$P_X$ orthogonal projector onto subspace $X$.\\
 -$I_X$ identity map on the set $X$.\\
 -For a closed linear map $R$ between   Hilbert spaces, $R^\star$ is  its adjoint. \\
 -$\Phi(x)=P_W \pi (x)P_W$ spherical function attached to the lowest $K$-type $W$ of $L_\cdot^G$.  $K_\tau(y,x)=d(L_\cdot^G)\Phi(x^{-1}y)$.\\
 -$M_{K-fin} (resp M^\infty, M^{H-\infty}) $ $K-$finite vectors in $M$ (smooth vectors in $M$, $H$-smooth vectors in $M$).\\
 -$dg,dh$ Haar measures on $G$, $H$.\\
 -A unitary representation is {\it square integrable}, equivalently a {\it Discrete Series} representation,  (resp. {\it integrable}) if some nonzero  matrix coefficient is square integrable (resp. integrable) with respect to Haar measure on the group in question. \\
 -$\Theta_{\pi_\mu^H}(...)$  Harish-Chandra character of the representation $\pi_\mu^H.$\\
 -$  M_{H-disc}$ is the closure of the linear subspace spanned by the totality of $H-$irreducible submodules. $ M_{disc}:= M_{G-disc}$\\
 -A representation $M$ is $H-${\it discretely decomposable} if $ M_{H-disc} =M.$\\
 -A representation is $H-${\it admissible} if it is $H-$discretely decomposable and each isotypic component is equal to a finite sum of $H-$irreducible representations.\\
 -$\mathcal U(\mathfrak g) $ (resp. $\mathfrak z(\mathcal U(\mathfrak g))=\mathfrak z_\mathfrak g$) universal enveloping algebra of the Lie algebra $\mathfrak g$ (resp. center of universal enveloping algebra).\\
 -$\mathrm{Cl}(X)= $closure of the set $X$.\\
 -$\mathbb T =S^1=SO(2)$ one dimensional torus.\\
-$S^{(r)}(V)$ the $r^{th}$-symmetric power of the vector space $V$.\\
-Cartan decomposition $Lie(G)=\g=\k+\p$, $\theta$ Cartan involution associated to $\k$,  $\t$ maximal abelian subalgebra for $\k$. \\ $\Phi(\g,\t)=\Phi(\g)=\Phi$ root  system attached to the pair $(\g_\C,\t_\C)$. Then,  either $\g_\alpha \subset \k_\C$ ($\alpha$ is compact) or $\g_\alpha \subset \p_\C$ ($\alpha$ is noncompact).\\
- $\Phi_c =\Phi(\k,\t)$ set of compact roots.\\
-$\Phi_n =\Phi(\p,\t)=\Phi_n(\g)=\Phi^n(\g)=\Phi_n^\g=\Phi_\g^n$ set of noncompact roots.\\
-For a system of positive roots $\Psi=\Psi_\g =\Psi (\g)$ in $\Phi(\g,\t)$, \\ -$\Psi_c:=\Psi(\k,\t):=\Psi \cap \Phi_c =\Psi \cap \Phi(\k,\t)$,\\
-$\Psi_n:=\Psi_n(\g):=\Psi_n^\g:=\Psi \cap \Phi_n =\Psi \cap \Phi(\p,\t)=\Psi \cap \Phi_n(\g)=\Psi \cap \Phi^n(\g)$.\\
-$\sigma$ involution in $\g$ that commutes with $\theta$. $\h=\{ X\in \g :\sigma(X)=X \}$, $\q=\{ X\in \g :\sigma(X)=-X \}$, $\g=\h+\q$, $\l=\h\cap \k$, $\h_0=\l +\q\cap \p$, $\u=\t \cap \k$.\\
-Similar notation for the pairs $(\h,\u), (\h_0,\u)$

\smallskip
 \textbf{Acknowledgements} The  authors would like to thank T. Kobayashi for much insight  and inspiration on the problems considered here. Also, we thank Michel Duflo, Birgit Speh, Yosihiki Oshima and Jan Frahm for conversations on the subject.

 The second author   thanks  Aarhus University for generous support, its hospitality and  excellent  working conditions during the preparation of this paper.

  Last but not least, we deeply thank to American Institute of Math and especially the program  Representation Theory   Noncommutative Geometry,
an AIM Research Community, for letting us to profit of their internet facilities, Seminars, Zoom. We like to point the excellent work
 of  the organizers of RTNG:  Birgit Speh, Nigel Higson, Pierre Clare, Angela Pasquale, Monica Nevins, Haluk Seng\"{u}n.

 \providecommand{\MR}{\relax\ifhmode\unskip\space\fi MR }
 \providecommand{\MRhref}[2]{%
   \href{http://www.ams.org/mathscinet-getitem?mr=#1}{#2}
 }
 \providecommand{\href}[2]{#2}


\begin{thebibliography}{10}
      \bibitem[BB]{BB} Baily, W., Borel, A..
  Compactification of arithmetic quotients of bounded symmetric domains.
  Ann. Math. (2) 84, 442-528 (1966).
  \bibitem[DES]{DES} Davison, M.,  Enright, T.,  Stanke, R.,  Covariant differential operators, Math. Ann. 288, 731-739 (1980).
       \bibitem[vDP]{vDP} van Dijk, G.; Pevzner, M.,
  Ring structures for holomorphic Discrete Series and Rankin-Cohen brackets.
  J. Lie Theory 17, No. 2, 283-305 (2007).
   \bibitem[DV]{DV}  Duflo, M., Vargas,  J.,  Branching laws for square integrable representations, Proc. Japan Acad. Ser. A Math. Sci. \textbf{86}, n 3, 49-54 (2010).
       \bibitem[DVer]{DVer} Duflo, M., Vergne, M.,
  Kirillov’s formula and Guillemin-Sternberg conjecture.
  C. R., Math., Acad. Sci. Paris 349, No. 23-24, 1213-1217 (2011).
       \bibitem[FK]{FK} Faraut J., Kaneyuki S., Koranyi A., Lu Q.k., Roos G., Analysis and geometry on complex homogeneous
  domains, Progress in Mathematics, Vol. 185, Birkhauser Boston, Inc., Boston, MA, 2000.
  \bibitem[GW]{GW} Gross, B., Wallach, N.,
  Restriction of small Discrete Series representations to symmetric subgroups.  The mathematical legacy of Harish-Chandra (Baltimore, MD, 1998),
  Proc. Sympos. Pure Math., 68, 255 - 272, (2000).

  \bibitem[HHO]{HHO}  Harris, B., He, H., \'O{}lafsson, G.,  The continuous spectrum in Discrete Series branching laws. Internat. J. Math. 24   no. 7, 1-29, (2013).
  \bibitem[Hel]{Hel}  Helgason, S., Groups and Geometric Analysis, Academic Press (20?).
  \bibitem[Hi]{Hi}Hilgert, J., Reproducing Kernels in Representation Theory. Symmetries in Complex Analysis, B. Gilligan, G. Roos eds.  Contemp. Math. 468, 1-98, (2008).
  \bibitem[JV]{JV} Jakobsen, H., Vergne,  M.: Restrictions and Expansions of Holomorphic Representations J. Funct. Anal.\textbf{ 34}, 29--53 (1979).
      \bibitem[Ki]{Ki} Kitagawa, M., Construction of irreducible $\mathcal U(\g)^{\g^\prime}$-modules and discretely decomposable restriction. arXiv:2410.1725v2[math.RT] (2024).
  \bibitem[KW]{KW} Knapp, A., Wallach, N., Szego Kernels Associated with Discrete Series representations. Invent. Math.  34, 163-200 (1976)
  \bibitem[Kob1]{Kob1} Kobayashi, T., Discrete decomposability of the restriction of $A_{\mathfrak q}(\lambda) $ with respect to reductive subgroups and its applications, Invent. Math. {\bf 117}, 181-205 (1994).

  \bibitem[Kob2]{Kob2} Kobayashi, T., Discrete decomposability of the restriction of $A_{\mathfrak q}(\lambda) $ with respect to reductive subgroups III, Invent. Math. {\bf 131}, 229-256 (1998).

    \bibitem[Kob3]{Kob3}  Kobayashi. T.. Multiplicity-free theorems of the restrictions of unitary highest weight modules with respect to reductive symmetric pairs. In Representation Theory and Automorphic Forms, volume 255 of Progr. Math., pages 45--109. Birkhäuser, 2007.


1965 \bibitem[Kob4]{Kob4} Kobayashi, T.,  A program for branching problems in the representation theory of real reductive groups,    Representations of Reductive Groups: In Honor of the 60th Birthday of David A. Vogan, Progress in Mathematics, Editors Nevins, M., Trapa, P.,  Vol.  312  277-322 (2015).
  \bibitem[KO]{KO}  Kobayashi, T., Oshima, Y., Classification of discretely decomposable $ A_\mathfrak q(\lambda)$ with respect to reductive symmetric pairs. Adv. Math. 231 (2012)  2013–2047.


  \bibitem[KP1]{KP1}	  Kobayashi T.,   Pevzner, M.,  Differential symmetry breaking operators. I. General theory and F-method. Selecta Mathematica (N.S.), 22(2) (2016) 801-845.
   \bibitem[KP2]{KP2}	  Kobayashi T.,   Pevzner, M.,  Differential symmetry breaking operators. II. Rankin--Cohen operators for symmetric pairs. Selecta     Mathematica (N.S.), 22(2):847-911, (2016).

  \bibitem[Na]{Na} Nakahama, R.,  Construction of intertwining operators between holomorphic Discrete Series representations. SIGMA Symmetry Integrability Geom. Methods Appl. 15 (2019), 036, 101 pages.
      \bibitem[Neeb]{Neeb} Neeb, K., An Introduction to Unitary Representations of
  Lie Groups, https://en.www.math.fau.de/lie-groups/scientific-staff/prof-dr-karl-hermann-neeb/lecture-notes/
    \bibitem[Neeb2]{Neeb2} Neeb, K., Holomorphy and Convexity in Lie Theory, De Gruyter expositions in Mathematics, Vol. 28, (1999).

 \bibitem[OS]{OS} {\O}rsted, B., Speh, B.,Branching laws for some unitary representations of
SL(4,R). SIGMA Symmetry Integrability Geom. Methods Appl. 4 (2008).
  \bibitem[OV]{OV} {\O}rsted, B., Vargas, J., Restriction of Discrete Series representations (Discrete spectrum), Duke Math. Journal, Vol. 123, 609-631, (2004).
      \bibitem[OV2]{OV2} {\O}rsted, B., Vargas, J., Branching problems in reproducing  kernel spaces, Duke Math. Journal
Vol. 169, No. 18,  2020 DOI 10.1215/00127094-2020-0032
   \bibitem[OV3]{OV3} {\O}rsted, B., Vargas, J., Pseudo-dual pairs and branching of Discrete Series,   Symmetry in Geometry and Analysis, Vol 2,  A Festschrift honoring Toshiyuki Kobayashi. Eds. M. Pevzner, H. Sekiguchi,  Progress in Math. Vol. 358 to appear Dec. 2024

   \bibitem[Pa]{Pa} Paradan, P.
  Kirillov’s orbit method: the case of Discrete Series representations.
  Duke Math. J. 168, No. 16, 3103-3134 (2019).
  \bibitem[Sa]{Sa} Satake, I.,  Algebraic structures of symmetric domains, Kano Memorial Lectures, Vol. 4, Iwanami Shoten,
  Tokyo, Princeton University Press, Princeton, N.J., 1980.
  \bibitem[Sch]{Sch} Schmid, W., Homogeneous complex manifolds and representations of semisimple Lie groups. Dissertation, University of California, Berkeley, CA, 1967. Math. Surveys Monogr., 31, Representation theory and harmonic analysis on semisimple Lie groups, 223–286, Amer. Math. Soc., Providence, RI, (1989).
        \bibitem[Th]{Th}
Thieleker, E.: On the quasisimple irreducible representation's of the Lorentz groups.   Trans. Am. Math. Soc. {\bf 179}, 465--505 (1973)
  \bibitem[Varad]{Varad} Varadarajan, V., Lie Groups, Lie Algebras, and Their Representations, Graduate Texts in Mathematics, Springer Verlag, (1994).
          \bibitem[Va]{Va} Vargas, J.,
  Associated symmetric pair and multiplicities of admissible restriction of Discrete Series,
     Int. J. Math. \textbf{27}   12 (2016). https://doi.org/10.1142/S0129167X16501007
         \bibitem[Wal]{Wal} Wallach, N., Harmonic Analysis on Homogeneous Spaces, Second
Edition, Dover Publications, Inc., Mineola, 2018.
  \bibitem[Wa1]{Wa1} Wallach, N., Real reductive groups I, Academic Press (1988).
  \bibitem[Wo]{Wo} Wong, Hon-Wai, Dolbeault cohomological realization of Zuckerman modules associated with finite rank representations. J. Funct. Anal. 129, no. 2, 428–454 (1995).
  \end{thebibliography}
  \end{document}